\documentclass[11pt]{elsarticle}
\usepackage{lipsum}
\def\ps@pprintTitle{
 \let\@oddhead\@empty
 \let\@evenhead\@empty
 \def\@oddfoot{}
 \let\@evenfoot\@oddfoot}
 
 \makeatletter
\def\ps@pprintTitle{%
  \let\@oddhead\@empty
  \let\@evenhead\@empty
  \let\@oddfoot\@empty
  \let\@evenfoot\@oddfoot
}
\makeatother

\renewcommand{\MaketitleBox}{
  \resetTitleCounters
  \def\baselinestretch{1}
  \begin{center}
    \def\baselinestretch{1}
    \Large \@title \par
    \vskip 18pt
    \normalsize\elsauthors \par
    \vskip 10pt
    \footnotesize \itshape \elsaddress \par
  \end{center}
  \vskip 12pt
}

\usepackage{amssymb}
\usepackage{amsthm}
\usepackage{lineno}
\usepackage{amsfonts}
\usepackage{amsmath}
\usepackage{graphicx}
\usepackage{cases}
\usepackage{epstopdf}
\usepackage{caption}
\usepackage{algorithmic}
\ifpdf
  \DeclareGraphicsExtensions{.eps,.pdf,.png,.jpg}
\else
  \DeclareGraphicsExtensions{.eps}
\fi

\usepackage{amssymb}
\usepackage{mathtools}
\usepackage{amsmath}
\usepackage{xcolor} 
\usepackage[mathscr]{eucal}
\usepackage{color}
\usepackage{geometry}
\usepackage{tikz}
\usetikzlibrary{shapes,arrows}
\usetikzlibrary{decorations.pathreplacing,angles,quotes}

\usepackage{todonotes}
\usepackage{dsfont}
\usepackage{cleveref}
\usepackage{verbatim}
\usepackage{graphicx}
\usepackage{subfig}
\usepackage{tikz}
\usepackage{pgfplots}
\usepgfplotslibrary{groupplots}

\newcommand{\mb}{\mathbb}

\newcommand{\mc}{\mathcal}
\newcommand{\wh}{\widehat}
\newcommand{\br}{\breve}

\newcommand{\mbR}{{\mathds{R}}}
\newcommand{\Id}{{\mathbb{I}}}

\newcommand{\mfK}{{\mathfrak{K} }}

\newcommand{\hQ}{{\widehat{Q}}}
\newcommand{\hH}{{\widehat{H}}}
\DeclareMathOperator{\tr}{tr}

\newtheorem{theorem}{Theorem}
\newtheorem{lemma}[theorem]{Lemma}
\newtheorem{proposition}[theorem]{Proposition}
\newtheorem{corollary}[theorem]{Corollary}

\newtheorem{remark}{Remark}
\newtheorem{assumption}{Assumption}

\begin{document}
\begin{frontmatter}

\title{LQG Risk-Sensitive Single-Agent and Major-Minor Mean-Field Game Systems: A Variational Framework \footnote{This work was presented at the IMSI Distributed Solutions to Complex Societal Problems Reunion Workshop held in Chicago, US, from February 20th to 24th, 2023, and at the SIAM Conference on Financial Mathematics and Engineering held in Philadelphia, US, from June 6th to 9th, 2023.}
\footnote{Dena Firoozi and Michèle Breton would like to acknowledge the support of the Natural Sciences and Engineering Research Council of Canada
(NSERC), grants RGPIN-2022-05337 and RGPIN-2020-05053. H. Liu would like to acknowledge the support of GERAD and IVADO through, respectively,  Doctoral Scholarship 2023-2024, and the NSERC-CREATE Program on Machine Learning in Quantitative Finance and Business Analytics (Fin-ML CREATE).   
}}
\author[1]{Hanchao Liu}
\author[1]{Dena Firoozi}
\author[1]{Michèle Breton}
\address[1]{Department of Decision Sciences, HEC Montréal, Montréal, QC, Canada\\ (email: hanchao.liu@hec.ca, dena.firoozi@hec.ca, michele.breton@hec.ca)}

\begin{abstract}
We develop a variational approach to address risk-sensitive optimal control problems with an exponential-of-integral cost functional in a general linear-quadratic-Gaussian (LQG) single-agent setup, offering new insights into such problems. Our analysis leads to the derivation of a nonlinear necessary and sufficient condition of optimality, expressed in terms of martingale processes. Subject to specific conditions, we find an equivalent risk-neutral measure, under which a linear state feedback form can be obtained for the optimal control. It is then shown that the obtained feedback control is consistent with the imposed condition and remains optimal under the original measure. Building upon this development, we (i) propose a variational framework for general LQG risk-sensitive mean-field games (MFGs) and (ii) advance the LQG risk-sensitive MFG theory by incorporating a major agent in the framework. The major agent interacts with a large number of minor agents, and unlike the minor agents, its influence on the system remains significant even with an increasing 
 number of minor agents. We derive the Markovian closed-loop best-response strategies of agents in the limiting case where the number of agents goes to infinity. We establish that the set of obtained best-response strategies yields a Nash equilibrium in the limiting case and an $\varepsilon$-Nash equilibrium in the finite-player case. 
\end{abstract}


\end{frontmatter}

\section{Introduction}\label{sec_intro}
The concept of risk-sensitive optimal control was introduced in \citet{jacobson1973optimal} within a  linear-quadratic-Gaussian (LQG) framework. In these problems, the agent's utility is described by an exponential function of the total cost it incurs over time. Since their inception, risk-sensitive control problems have captured considerable interest in the literature. Notably, the theory has been extended to encompass nonlinear risk-sensitive problems (\citet{kumar1981optimal} and  \citet{nagai1996bellman}) and imperfect information (\citet{pan1996model}), leading to a broader understanding of these systems. Furthermore, different methodologies have been developed, each offering unique insights to address such problems (see, for example, \citet{duncan2013linear} and \citet{lim2005new}). \citet{bacsar2021robust} provides an extensive overview of the literature on this topic. 

Mean-field game (MFG) theory concerns the study and analysis of dynamic games involving a large number of indistinguishable agents who are asymptotically negligible.  In such games, each agent is weakly coupled with the empirical distribution of other agents whose mathematical limit, as the number of agents goes to infinity, is called the mean-field distribution. In such games, the behavior of players in large populations, along with the resulting equilibrium, may be approximated by the solution of infinite-population games (see e.g. \citet{lasry2007mean, caines2021mean,carmona2018probabilistic,bensoussan2013mean,carmona2013mean,cardaliaguet2019master}).

A notable advancement of MFG theory involves the integration of the so-called \emph{major} agents within the established framework. Unlike minor agents, whose impact decreases as the number of agents increases, the impact of a major agent is not negligible and does not collapse when the size of the population tends to infinity 
(\citet{huang2010large,firoozi2020convex,carmona2016probabilistic}). Various interpretations of such systems have been proposed. In the area of investment finance, for instance, one can consider that institutional and private investors' decisions do not have a commensurable impact on the market (see e.g. \cite{FirooziISDG2017,huang2019mean}).

MFGs find relevance in various domains, and particularly within financial markets, they emerge as natural modeling choices for addressing a wide array of issues. Notably, applications have been proposed in systemic risk (\citet{carmona2013mean,Garnier2013LargeDeviationsMean,Bo2015SystemicRiskInterbanking, chang2022Systemic}), price impact and optimal execution (\citet{casgrain_meanfield_2020, FirooziISDG2017,Cardaliaguet2018Meanfieldgame, carmona2015probabilistic,huang2019mean}), cryptocurrencies (\citet{li2023mean}), portfolio trading (\citet{lehalle2019mean}), equillibrium pricing (\citet{Firoozi2022MAFI,gomes_mean-field_2018,fujii_mean_2021}), and market design (\citet{shrivats2021principal}). 
In economics and finance, it is well recognized that attitude toward risk, or \emph{risk sensitivity}, plays an important role in the determination of agents' optimal decisions or strategies (\citet{bielecki2000risk,bielecki2003economic,fleming2000risk}). 

The study of risk-sensitive models is crucial as risk-neutral models often fall short in capturing all the behaviors observed in reality. This consideration is especially pertinent in many economic and financial contexts as risk sensitivity, and its disparity among players, needs to be accounted for when characterizing equilibrium strategies. This is also the case in the area of mean-field games, where recent developments were proposed to address risk-sensitive MFGs (\citet{tembine2013risk, saldi2018discrete,saldi2022partially,moon2016linear,moon2019risk}).

In this paper, we develop a variational approach to address risk-sensitive optimal control problems with an exponential-of-integral cost functional in a general LQG single-agent setup, drawing inspiration from  \citet{firoozi2020convex}. Our analysis leads to the derivation of a nonlinear necessary and sufficient condition of optimality, expressed in terms of martingale processes. Subject to specific conditions we find an equivalent risk-neutral measure, under which a linear state feedback form can be obtained for the optimal control. It is then shown that the obtained feedback control is consistent with the imposed condition and remains optimal under the original measure. Building upon this development, we (i) propose a variational framework for general LQG risk-sensitive MFGs and (ii) extend the LQG risk-sensitive MFG theory by incorporating a major agent in the framework. We derive the Markovian closed-loop best-response strategies of agents in the limiting case where the number of agents goes to infinity. We establish that the set of obtained best-response strategies yields a Nash equilibrium in the limiting case and an $\varepsilon$-Nash equilibrium in the finite-player case. We conclude the paper by presenting illustrative numerical experiments.

More specifically, the contributions of this work include the following: 
\begin{itemize}
     \item We develop a variational approach to solve risk-sensitive optimal control problems. This approach offers new perspectives on the inherent nature of risk-sensitive problems, distinct from existing methodologies. Specifically: (i) it demonstrates how the nonlinearity of exponential risk-sensitive cost functionals can be translated into a necessary and sufficient condition for optimality involving a quotient of martingale processes; 
     (ii) by establishing an equivalent risk-neutral measure in terms of the state and control processes, it explains the connection between the model and its counterpart under the risk-neutral measure. This technique enables the derivation of explicit solutions, even in the presence of the nonlinear term within the necessary and sufficient condition of optimality that involves the state and control processes; (iii) it allows a deeper understanding of risk-sensitivity's implications by  effectively tracing the impact of risk-sensitivity on the propagation of policy perturbations throughout the system; and (iv) it extends the applicability of variational analysis to the context of risk-sensitive problems with exponential cost functionals and facilitates the characterization of optimal strategies for complex or nonclassical setups.

    \item We advance the theory of LQG risk-sensitive MFGs by incorporating a non-negligible major agent, whose impact on the system remains significant even as the number of minor agents goes to infinity. To the best of our knowledge, the literature has not yet explored such MFGs. In particular, our work stands apart from \citet{ChenESAIM2023}, where a risk-sensitive MFG involving a group of asymptotically negligible major agents is considered and the average state of this group appears in the model. Unlike that scenario, our model features the significant influence of the major agent in the limiting model. 
      
   \item Building upon the the aforementioned developments, we propose a variational framework for general LQG risk-sensitive MFGs with a major agent. Through this framework, we gain valuable insights into the interplay among risk sensitivity, the major agent and minor agents within the context of LQG MFGs. Specifically, it illustrates: (i) the impact of the major agent's risk sensitivity on the propagation of its policy perturbations throughout the system (i.e. on individual minor agents), and hence the formation of the aggregate effect of minor agents (i.e. the mean field) and the equilibrium, and (ii) the impact of a representative minor agent's risk sensitivity on the propagation of its policy perturbations across the system (i.e. on other minor agents and the major agent) and hence the formation of the the mean field and the equilibrium. More precisely, using this variational approach, we derive the Nash equilibrium under the infinite-population setup. Then, we establish that the equilibrium strategies for the infinite-population case lead to an $\varepsilon$-Nash equilibrium for the finite-population game. Our proof of the $\varepsilon$-Nash property differs from existing ones in the literature on MFGs with a major agent (\citet{huang2010large, carmona2016probabilistic}) and on risk-sensitive MFGs (\citet{moon2019risk,moon2016linear}), leveraging the specific conditions of the model under investigation. We further investigate and provide insights into the impact of agents' risk sensitivity on the equilibrium by conducting a comparative analysis with the risk-neutral case. 
  \end{itemize}

This paper is organized as follows. Section \ref{sec_lit} reviews the relevant literature. Section \ref{sec_control} outlines the variational approach 
developed to solve LQG risk-sensitive single-agent optimal control problems, which will be used to characterize the best response of MFG agents in the subsequent section. Section \ref{sec_game} employs the developed variational analysis to the LQG risk-sensitive MFGs with major and minor agents, in order to obtain the Markovian closed-loop Nash equilibrium of infinite-population MFGs, and it shows that this Nash equilibrium provides an approximate equilibrium for the finite-population game. Section \ref{sec_conclusion} provides a brief conclusion.

\section{Literature Review}\label{sec_lit}
\subsection{Risk-Sensitive Single-Agent and Mean-Field Game Systems}\label{sec_lit_1}

Risk-sensitive optimal control problems were introduced in \cite{jacobson1973optimal} in a finite horizon LQG setting, where the agent's utility is an exponential function of its total cost over time. 
Later, \cite{pan1996model} studied linear singularly perturbed systems with long-term time-average exponential quadratic costs under perfect and noisy state measurements. 
A generalization to the case where costs are not quadratic functions was proposed in \cite{kumar1981optimal} and \cite{nagai1996bellman}, and infinite horizon risk-sensitive optimal control problems in discrete-time were addressed in \cite{whittle1981risk}. More recently, an alternative approach based on the first principles was suggested in \cite{duncan2013linear}, providing additional insights into the solution of such systems. 
This method was subsequently applied in \cite{duncan2015linear} to characterize the Nash equilibrium of a two-player noncooperative stochastic differential game. Additionally, maximum principle for risk-sensitive control was established in \cite{lim2005new}. Furthermore, a risk-sensitive LQ optimal
control problem involving a large number of agents with  mean field interactions was explored in \cite{wang2024large}.
 
Mean-field games have recently been extended to a risk-sensitive context. A general setup for risk-sensitive MFGs was proposed in \cite{tembine2013risk}. 
Later, \cite{moon2016linear} studied a MFG involving heterogeneous agents with linear dynamics and an exponential quadratic integral cost using the Nash certainty equivalence (NCE) method (see \cite{huang2006large}). 
Recently, \cite{moon2019risk} developed a stochastic maximum principle for risk-sensitive MFGs over a finite horizon. Similar results were obtained for the discrete-time setup  \cite{saldi2018discrete}. Additionally, \cite{Minyi-Wang2025risk} developed a direct approach to LQ risk-sensitive mean-field games. A risk-sensitive MFG involving two large-population groups of agents with different impacts on the system was studied in \cite{ChenESAIM2023}, where the average state of each group appears in the model. Risk-sensitive MFGs with common noise were explored and applied to systemic risk in interbank markets in \cite{FirooziRen-2024-common-noise}.     

Other research avenues that are closely related to risk-sensitive MFGs include robust MFGs (see e.g. \citet{HuangRobust2017, HuangJaimungal2017}), where model ambiguity is incorporated into the optimization problem.

\subsection{MFGs with Major and Minor Agents}
An advancement of the MFG theory incorporates the interaction between major and minor agents, and studies (approximate) Nash equilibria between them. To address major-minor MFGs, various approaches have been proposed, including NCE (\citet{huang2010large}, \citet{nourian2013epsilon}), probabilistic approaches involving the solution of forward-backward stochastic differential equations (FBSDE) (\citet{carmona2016probabilistic}, \citet{carmona2017alternative}), asymptotic solvability (\citet{huang2019linear}), master equations (\citet{lasry2018mean}, \citet{cardaliaguet2019master}), and convex analysis (\citet{firoozi2020convex}), the last of which being developed under the LQG setup. 
These approaches have been shown to yield equivalent Markovian closed-loop solutions (\citet{huang2020linear}, \citet{firoozi2022lqg}) in the limit. Major-minor MFGs have been developed in various directions, including scenarios involving partial observations (\cite{sen2016mean,firoozi2020epsilon}) and those with switching and stopping strategies (\cite{FIROOZI2022-hybrid}). Another line of research characterizes a Stackelberg equilibrium between the major agent and the minor agents; see e.g. \citet{BensoussanSICON2017,BasarMoon2018}.

In the next section, we develop a variational analysis for LQG risk-sensitive optimal control problems. This analysis will serve as the foundation for addressing major-minor MFG systems, which will be discussed in detail in the subsequent section.

\section{Variational Approach to LQG Risk-Sensitive Optimal Control Problems}
\label{sec_control}
We begin by examining (single-agent) risk-sensitive optimal control problems under a general linear-quadratic-Gaussian  framework. This approach will allow us to streamline the notation and enhance the clarity of the subsequent expositions. Building on the variational analysis method proposed in \citet{firoozi2020convex}, we extend the methodology by using a change of measure technique to derive the optimal control actions for LQG risk-sensitive problems with exponential cost functionals. These results will then be used in the subsequent section to determine the best-response strategies of both the major agent and a representative minor agent.

We consider a general LQG risk-sensitive model with dynamics given by  
\begin{equation} \label{dynamic}
dx_{t}=(Ax_{t}+Bu_{t}+b(t))dt+\sigma (t)dw_{t}
\end{equation}
where $x_{t}\in \mathbb{R}^n$ and $u_{t}\in \mathbb{R}^{m}$ are respectively the state and control vectors at $t$ and $w_{t}\in \mathbb{R}^{r}$ is a standard $r$-dimensional Wiener process defined on a given filtered probability space $\left ( \Omega,\mathcal{F},\{\mathcal{F}_{t}\}_{t\in \mc{T}},\mathbb{P} \right )$,   $\mc{T}=\left[ 0,T \right]$ with $T>0$ fixed. Moreover, we define $A$ and $B$ to be constant matrices with compatible dimensions, and $b(t)$ and $\sigma(t)$ to be deterministic continuous functions on $\mc{T}$. 

The risk-sensitive cost functional $J(u)$ to be minimized is given by 
\begin{equation} \label{cost}
J(u)=\mathbb{E}\left[ \exp{\left(\delta \Lambda_{T}(u)\right)} \right], \end{equation}
where
\begin{equation}
 \Lambda_{T}(u):= \frac{1}{2} \int_{0}^{T}\Big(\left<Qx_{s},x_{s} \right> +2\left<Su_{s},x_{s} \right>+\left<Ru_{s}, u_{s}\right>-2\left<\eta ,x_{s}  \right>-2\left<\zeta ,u_{s} \right> \Big)  ds+\frac{1}{2}\left< \widehat{Q}X_{T},X_{T}\right>.\label{costG}
\end{equation}

All the parameters ($Q,S,R,\eta, \zeta, \widehat{Q}$) in the above cost functional are vectors or matrices of an appropriate dimension. The positive scalar constant $\delta$ represents the degree of risk sensitivity, with $0<\delta< \infty$ modeling a risk-averse behavior. It is worth noting that the risk-neutral
    cost functional $\mathbb{E}\left[ {\Lambda_{T}(u)} \right]$ can be seen as the limit of the risk-sensitive cost functionals $\frac{1}{\delta}\log\mathbb{E}\left [ \mathrm{exp}\left (\delta \Lambda_{T}(u) \right ) \right ]$ or $\frac{1}{\delta}\mathbb{E}\left [ \mathrm{exp}\left (\delta \Lambda_{T}(u) \right ) - 1 \right ]$ when $\delta \rightarrow 0$. Both of these risk-sensitive cost functionals yield the same optimal control action as the cost functional \eqref{cost}-\eqref{costG}, due to the strictly increasing property of the logarithm and linear functions, respectively. For notational convenience, our analysis will focus on the cost functional \eqref{cost}-\eqref{costG}.

The following assumption provides the conditions under which the cost functional $\Lambda_{T}(u)$ is strictly convex. 
\begin{assumption}\label{convexity} $R>0$, $\widehat{Q}\geq 0$, and $Q-SR^{-1}S^{\intercal } \geq 0$.
\end{assumption}

The filtration $\mc{F}=(\mc{F}_t)_{t\in \mc{T}}$, with $\mathcal{F}_{t}:=\sigma (x_{s};0\leq s\leq t)$, which is the $\sigma$-algebra generated by the process $x_{t}$, constitutes the information set of the agent. Subsequently, the admissible set $\mc{U}$ of control actions is the Hilbert space  consisting of all $\mathcal{F}$-adapted $\mathbb{R}^{m}$-valued processes such that $\mathbb{E}\left [ \int_{0}^{T} \left \| u_{t} \right \|^{2}dt\right ]< \infty $, where $\left \| . \right \|$ is the corresponding Euclidean norm. 
{\begin{proposition} \label{strict_convexity}
Suppose \Cref{convexity} holds. Then the cost functional \eqref{cost}-\eqref{costG} is strictly convex.
\end{proposition}
\begin{proof}
By applying the standard method of completing the square to $\Lambda_{T}(u)$, given by \eqref{costG}, we obtain
\begin{align}
 \Lambda_{T}(u):= \frac{1}{2} \int_{0}^{T}\Big(&\left<\tilde{Q}x_{s},x_{s} \right> +\left<R(u_{s}-R^{-1}S^{\intercal}x_{s}),u_{s}-R^{-1}S^{\intercal}x_{s}\right>-2\left<\zeta ,u_{s}-R^{-1}S^{\intercal}x_{s} \right> \notag\\ &-2\left<\eta ,x_{s}  \right>+2\left<\zeta ,R^{-1}S^{\intercal}x_{s}) \right>\Big)  ds+\frac{1}{2}\left< \widehat{Q}X_{T},X_{T}\right>. \label{cost-3}  
\end{align}
Let $\widetilde{u}$ and $\widehat{u}$ be two different elements in $\mc{U}$ such that 
\begin{equation}
 \mathbb{E}\left [ \int_{0}^{T} \left \| \widetilde{u}_t- \widehat{u}_t\right \|^{2}dt\right ] > 0.   \label{diff-cont}
\end{equation}
Then we have
\begin{equation} \label{concon}
\mathbb{E}\left [ \int_{0}^{T} \left \| \widetilde{u}_t- \widehat{u}_t-R^{-1}S^{\intercal}(\widetilde{x}_s-\widehat{x}_s)\right \|^{2}dt\right ] > 0. 
\end{equation}
To verify the validity of the above statement, we proceed by proof by contraposition. Suppose \eqref{concon} does not hold, i.e, $\mathbb{E}\left [ \int_{0}^{T} \left \| \widetilde{u}_t- \widehat{u}_t-R^{-1}S^{\intercal}(\widetilde{x}_s-\widehat{x}_s)\right \|^{2}dt\right ] = 0$, or equivalently $\widetilde{u}_{t}-R^{-1}S^{\intercal}\widetilde{x}_{t}=\widehat{u}_{t}-R^{-1}S^{\intercal}\widehat{x}_{t},\, \mathbb{P} \times \mu$-almost everywhere, where $\mu$ stands for the Lebesgue measure. In this case, from \eqref{dynamic}, $\widetilde{x}_t-\widehat{x}_t$ satisfies 
\begin{equation}
d(\widetilde{x}_t-\widehat{x}_t)=[(A +BR^{-1}S^{\intercal})(\widetilde{x}_t-\widehat{x}_t)]dt,
\end{equation}
with $\widetilde{x}_0-\widehat{x}_0=0$. Thus, $\widetilde{x}_t=\widehat{x}_t,\, \mathbb{P} \times \mu$-almost everywhere, which leads to 
\begin{equation} \label{concon1}
\mathbb{E}\left [ \int_{0}^{T} \left \| \widetilde{u}_t- \widehat{u}_t-R^{-1}S^{\intercal}(\widetilde{x}_s-\widehat{x}_s)\right \|^{2}dt\right ]= \mathbb{E}\left [ \int_{0}^{T} \left \| \widetilde{u}_t- \widehat{u}_t\right \|^{2}dt\right ]=0.
\end{equation}
Thus, \eqref{concon} holds for any $\widetilde{u} \in \mc{U}$ and $\widehat{u} \in \mc{U}$ satisfying \eqref{diff-cont}. We then define the following sets  
    \begin{align}
     &E_1:=\left\{\omega \in \Omega: \int_{0}^{T}(\left\|\widetilde{x}_t \right\|^{2}+\left\|\widehat{x}_t\right\|^{2}+\left\|\widetilde{u}_t \right\|^{2}+\left\|\widehat{u}_t \right\|^{2})dt < \infty  \right\}, \notag \\ 
     &E_2:=\left \{\omega \in \Omega: \int_{0}^{T} \left \| \widetilde{u}_t- \widehat{u}_t-R^{-1}S^{\intercal}(\widetilde{x}_s-\widehat{x}_s)\right \|^{2}dt >0 \right \},\notag\\
     &E:=E_1 \cap E_2.\notag
    \end{align}
Since $P(E_1)=1$ and $P(E_2)>0$, it follows that $P(E)>0$. Hence, $\forall \omega \in E$, we have $\widetilde{u}_{t}-R^{-1}S^{\intercal}\widetilde{x}_{t} \neq \widehat{u}_{t}-R^{-1}S^{\intercal}\widehat{x}_{t}$ on a set within the Borel $\sigma$-algebra $\mathcal{B}(\mc{T})$ that has a positive Lebesgue measure. Subsequently, given that $R > 0$, we can easily show that $\forall \omega \in E$
    \begin{align}
        &\int_{0}^{T}  \lambda \left<R(\widetilde{u}_{t}-R^{-1}S^{\intercal}\widetilde{x}_{t}), \widetilde{u}_{t}-R^{-1}S^{\intercal}\widetilde{x}_{t}\right> +(1-\lambda)\left<R(\widehat{u}_{t}-R^{-1}S^{\intercal}\widehat{x}_{t}), \widehat{u}_{t}-R^{-1}S^{\intercal}\widehat{x}_{t}\right>dt  \allowdisplaybreaks\\  &>\!\!\int_{0}^{T}\!\! \Big<R(\lambda (\widetilde{u}_{t}\!-\!R^{-1}S^{\intercal}\widetilde{x}_{t})+(1-\lambda)(\widehat{u}_{t}\!-\!R^{-1}S^{\intercal}\widehat{x}_{t})),\lambda (\widetilde{u}_{t}\!-\!R^{-1}S^{\intercal}\widetilde{x}_{t})\notag  +(1-\lambda)(\widehat{u}_{t}\!-\!R^{-1}S^{\intercal}\widehat{x}_{t})\Big> dt\notag \allowdisplaybreaks\\  =\!\! &\int_{0}^{T}\!\! \Big<R(\lambda \widetilde{u}_{t }+(1-\lambda)\widehat{u}_{t}\!-\!R^{-1}S^{\intercal}(\lambda \widetilde{x}_{t}+(1-\lambda)\widehat{x}_{t})),\lambda \widetilde{u}_{t }+(1-\lambda)\widehat{u}_{t}\!-\!R^{-1}S^{\intercal}(\lambda \widetilde{x}_{t}+(1-\lambda)\widehat{x}_{t})\Big> dt.  \notag
      \end{align}
      Moreover, \Cref{convexity} guarantees that all other terms in \eqref{cost-3} are convex (not necessarily strictly convex). Hence, $\forall \omega \in E$, we can conclude that 
\begin{equation} \label{sc1}
\lambda \Lambda_{T}(\widetilde{u})+(1-\lambda)\Lambda_{T}(\widehat{u}) > \Lambda_{T}(\lambda \widetilde{u}+(1-\lambda)\widehat{u}).
\end{equation}
Furthermore, the exponential function is both strictly increasing and strictly convex. By leveraging this property in conjunction with \eqref{sc1}, $\forall \omega \in E$, we obtain
\begin{equation}
\lambda \exp{\left(\delta \Lambda_{T}(\widetilde{u})\right)}+(1-\lambda)  \exp{\left(\delta \Lambda_{T}(\widehat{u})\right)} > \exp \delta{\Lambda_{T}(\lambda \widetilde{u}+(1-\lambda)\widehat{u})}. \label{c}
\end{equation}
Therefore, we have 
\begin{equation} \label{e3}
\mathbb{E}\left [  (\lambda \exp{\left(\delta \Lambda_{T}(\widetilde{u})\right)}+(1-\lambda)  \exp{\left(\delta \Lambda_{T}(\widehat{u})\right)})\mathds{1}_{E} \right ] > \mathbb{E}\left [\exp \delta{\Lambda_{T}(\lambda \widetilde{u}+(1-\lambda)\widehat{u})}\mathds{1}_{E}\right], 
\end{equation}
where $\mathds{1}_{E}$ denotes an indicator function defined on $E$. Moreover, it is evident that, due to the convexity of $\Lambda_{T}$, we have $\forall \omega \in E^{c}$ 
\begin{equation}
\lambda \exp {\left(\delta \Lambda_{T}(\widetilde{u})\right)}+(1-\lambda)  \exp{\left(\delta \Lambda_{T}(\widehat{u})\right)} \geq \exp {\delta \Lambda_{T}(\lambda \widetilde{u}+(1-\lambda)\widehat{u})}.\label{c1}
\end{equation}
Hence, we obtain
\begin{equation} \label{ec3}
\mathbb{E}\left [  (\lambda \exp{\left(\delta \Lambda_{T}(\widetilde{u})\right)}+(1-\lambda)  \exp{\left(\delta \Lambda_{T}(\widehat{u})\right)})\mathds{1}_{E^{c}} \right ] \geq \mathbb{E}\left [\exp \delta{\Lambda_{T}(\lambda \widetilde{u}+(1-\lambda)\widehat{u})}\mathds{1}_{E^{c}}\right].
\end{equation}
Finally, the strict convexity of the cost functional given by \eqref{cost}-\eqref{costG} is established by adding together \eqref{e3} and \eqref{ec3}. 
\end{proof}}
The following theorem provides the Gâteaux derivative of the cost functional for the system described by \eqref{dynamic}-\eqref{cost}.
\begin{proposition}\label{thm:Cntrl_initial}
For any $u, \omega \in \mathcal{U}$, we have 
\begin{align} \label{gtder}
\lim_{\epsilon  \to 0}\frac{J\left ( u+\epsilon \omega  \right )-J\left ( u \right )}{\epsilon }
=\delta \mathbb{E}\Bigg[ \int_{0}^{T}\omega ^{\intercal}\Big[B^{\intercal}e^{-A^{\intercal}t}M_{2,t}+M_{1,t}(Ru_{t}+S^{\intercal}x_{t}-\zeta \notag\\\hspace{3cm}+B^{\intercal}\int_{0}^{t}e^{A^{\intercal}(s-t)}(Qx_{t}+Su_{t}-\eta)ds )\Big]dt\Bigg]
\end{align}
where
\begin{align}
M_{1,t}(u)&= \mathbb{E}\left [ e^{\delta \Lambda_{T}(u)}\Big | \mathcal{F}_{t} \right ] \label{M1t} \\
M_{2,t}(u)&= \mathbb{E}\left [ e^{\delta \Lambda_{T}(u)}(e^{A^{\intercal}T}\widehat{Q}x_{t}+\int_{0}^{T}e^{A^{\intercal}s}(Qx_{s}+Su_{s}-\eta)ds)\Big | \mathcal{F}_{t} \right ]. \label{M2t}
\end{align}
\end{proposition}

\begin{proof}
   First, the strong solution $x_{t}$ to the SDE \eqref{dynamic} under the control action $u_{t}$ is given by
\begin{equation}  \label{xt}
x_{t}=e^{At}x_{0}+\int_{0}^{t}e^{A(t-s)}(Bu_{s}+b(s))ds+\int_{0}^{t}e^{A(t-s)}\sigma (s)dw_{s}.    
\end{equation}
Subsequently, the solution $x^{\epsilon }_{t}$ under the perturbed control action $u_t^\epsilon:=u_{t}+\epsilon\omega_{t}$ is given by 
\begin{equation}  
x^{\epsilon }_{t}=e^{At}x_{0}+\int_{0}^{t}e^{A(t-s)}(Bu_{s}+b(s))ds+\int_{0}^{t}e^{A(t-s)}\sigma (s)dw_{s}+\epsilon \int_{0}^{t}e^{A(t-s)}B\omega_{s}ds.\label{x_eps}
\end{equation}
From \eqref{xt}-\eqref{x_eps}, we have 
 \begin{equation}
x^{\epsilon }_{t}=x_{t}+\epsilon \int_{0}^{t}e^{A(t-s)}B\omega_{s}ds.
\end{equation} 
By a direct computation, we obtain
\begin{align} \label{termi}
\left< \wh{Q}x^{\epsilon }_{T},x^{\epsilon }_{T}\right>\!-\! \left< \wh{Q}x_{T},x_{T}\right>=2\epsilon\!\left<\! \wh{Q}x_{T},\int_{0}^{T}\!\!\!\!e^{A(T-s)}B\omega_sds\!\right>\!+\!\epsilon^{2}\!\left<\! \wh{Q}\int_{0}^{T}\!\!\!\!e^{A(T-s)}B\omega_sds,\int_{0}^{T}\!\!\!\!e^{A(T-s)}B\omega_sds\!\right> \notag \\=2\epsilon\int_{0}^{T}\left< B^{\intercal}e^{A^{\intercal}(T-s)}\wh{Q}x_{T},\omega_s\right>ds+ \epsilon^{2}\left< \wh{Q}\int_{0}^{T}\!e^{A(T-s)}B\omega_sds,\int_{0}^{T}e^{A(T-s)}B\omega_sds\right>.
\end{align}
Finally, from \eqref{costG} and \eqref{xt}-\eqref{termi}, we can compute the difference 
\begin{align}\label{LambdaTdiff}
&\Lambda_{T}(u+\epsilon \omega )-\Lambda_{T}(u)
 =\epsilon \Bigg[\int_{0}^{T}\Bigg\{\left ( \int_{0}^{s}e^{A(s-t)}B\omega_tdt \right ) ^{\intercal}\left(Qx_{s}+Su_{s}-\eta\right)ds +\Big(\left< R\omega_{s},u_{s}\right> \\ &\hspace{0.3cm}+\left<S\omega_{s},x_{s} \right> -\left< \zeta ,\omega_{s}\right>+\left< B^{\intercal}e^{A^{\intercal}(T-s)}\wh{Q}x_{T},\omega_s\right> \Big) ds\Bigg\}\Bigg] \hspace{0cm} + \epsilon^{2}\Bigg[\int_{0}^{T}\Bigg\{\left ( \int_{0}^{s}e^{A(s-t)}B\omega_tdt \right ) ^{\intercal}\notag \allowdisplaybreaks\\ &\hspace{0.3cm}\times \Big(\notag\allowdisplaybreaks Q\int_{0}^{s}e^{A(s-t)}B\omega_tdt+S\omega_{s}\Big)+\left< R\omega_{s},\omega_{s}\right>\Bigg\}ds+\left< \wh{Q}\int_{0}^{T}e^{A(T-s)}B\omega_sds,\int_{0}^{T}e^{A(T-s)}B\omega_sds\right>\Bigg], \notag
\end{align}
which equivalently may be expressed as 
\begin{equation}
    \Lambda_{T}(u+\epsilon \omega )-\Lambda_{T}(u)
 = \epsilon \iota _1 +\epsilon^{2} \iota _2, \label{equiv-rep}
\end{equation}
where the random variables $\iota _1, \iota _2$ do not depend on $\epsilon$ and are $\mathbb{P}$-almost surely finite. Therefore, we have 
\begin{align} 
\left<\mathcal{D}J(u),\omega \right> =&\lim_{\epsilon  \to 0}\frac{J\left ( u+\epsilon \omega  \right )-J\left ( u \right )}{\epsilon }=\lim_{\epsilon  \to 0}\mathbb{E}\left [ e^{\delta \Lambda_{T}(u)}\left (\frac{e^{\delta \Lambda_{T}(u+\epsilon \omega )-\delta \Lambda_{T}(u)}-1}{\epsilon} \right) \right ]. \label{deri0}
\end{align}
By assuming the interchangeability of the limit and the expectation\footnote{See \Cref{comparison}.}, and applying L'Hôpital's rule as in 
\begin{equation}\label{deriv-0}
\lim_{\epsilon\rightarrow 0}\frac{e^{\delta \Lambda_{T}(u+\epsilon \omega )-\delta \Lambda_{T}(u)}-1}{\epsilon}=\lim_{\epsilon\rightarrow 0}\frac{e^{\delta \epsilon \iota _1 +\delta \epsilon^{2} \iota _2}-1}{\epsilon}=\delta  \iota_1,
\end{equation}
we obtain 
\begin{align} \label{deri1}
\left<\mathcal{D}J(u),\omega \right> =&\delta \mathbb{E}[e^{\delta \Lambda_{T}(u)}\iota_1]=\delta \mathbb{E}e^{\delta \Lambda_{T}(u)}\int_{0}^{T} \Bigg[\bigg\{\int_{0}^{s} \omega^{\intercal}_tB^{\intercal}e^{A^{\intercal}(s-t)}(Qx_{s}+Su_{s}-\eta )dt \\ &\hspace{2.4cm}+\omega ^{\intercal }_s\left (B^{\intercal}e^{A^{\intercal}(T-s)}\wh{Q}x_{T}+Ru_{s}+S^{\intercal }x_{s}-\zeta    \right )\bigg\}ds \Bigg]. \notag
\end{align}
After changing the order of the double integral using Fubini's theorem, \eqref{deri1} is equivalent to 
\begin{align} 
  \label{gateau1}
 \left<\mathcal{D}J(u),\omega \right> =&   \delta\mathbb{E}\Bigg[\int_{0}^{T}\omega ^{\intercal}_te ^{\delta \Lambda_{T}(u)}\Big[B^{\intercal}e^{A^{\intercal}(T-t)}\widehat{Q}x_{T}+Ru_{t}+S^{\intercal} x_{t}-\zeta   \allowdisplaybreaks\\ &+B^{\intercal}\int_{t}^{T}e^{A^{\intercal}(s-t)}(Qx_{s}+Su_{s}-\eta )ds \Big]dt\Bigg]  \notag \allowdisplaybreaks\\ =& \delta \mathbb{E}\Bigg[\int_{0}^{T}\omega ^{\intercal}_t\Big[B^{\intercal}e ^{\delta \Lambda_{T}(u)}(e^{A^{\intercal}(T-t)}\widehat{Q}x_{T}+\int_{0}^{T}e^{A^{\intercal}(s-t)}(Qx_{s}+Su_{s}-\eta )ds)\notag \allowdisplaybreaks\\&+e ^{ \Lambda_{T}(u)}\left(Ru_{t}+S^{\intercal} x_{t}-\zeta\right)-B^{\intercal}e ^{ \Lambda_{T}(u)}\int_{0}^{t}e^{A^{\intercal}(s-t)}(Qx_{s}+Su_{s}-\eta )ds\Big]dt\Bigg]. \notag
\end{align}
Finally, using the smoothing property of conditional expectations, \eqref{gateau1} can be written as in \eqref{gtder}, 
where $M_{1,t}(u)$ and $M_{2,t}(u)$ are defined by \eqref{M1t} and \eqref{M2t}.
\end{proof}
Since $J(u)$ is strictly convex, $u$ is the unique minimizer of the risk-sensitive cost functional if $\left<\mathcal{D}J(u),\omega \right>=0$ for all $\omega \in \mc{U}$ (see \citet{ciarlet2013linear}).
We observe that \eqref{gateau1} takes the form of an inner product. Note that the control action $u^{\circ}_t$ given by 
\begin{equation} \label{optc}
u^{\circ}_t =-R^{-1}\Bigg[S^{\intercal}x_{t}-\zeta+B^{\intercal}\left (e^{-A^{\intercal}t}\frac{M_{2,t}(u^{\circ})}{M_{1,t}(u^{\circ})}-\int_{0}^{t}e^{A^{\intercal}(s-t)}(Qx_{s}+Su^{\circ}_s-\eta)ds  \right )\Bigg]\ \end{equation}
makes $\left<\mathcal{D}J(u^{\circ}),\omega \right>=0$, 
$\mathbb{P}$-a.s., where the martingale $M_{1,t}(u^{\circ})$ characterized by \eqref{M1t} is almost surely positive. Therefore, $J(u)$ is Gâteaux differentiable at $u^{\circ}$ and $\left<\mathcal{D}J(u^{\circ}),\omega \right>=0$ 
for all $\omega \in \mc{U}$, making $u^{\circ}$ the unique minimizer of $J(u)$.
 
However, the current form of $u^{\circ}$ is not practical for implementation purposes. Therefore, our objective is to obtain an explicit representation for $u^{\circ}$. Nonetheless, due to the presence of the term $\frac{M_{2,t}(u^{\circ})}{M_{1,t}(u^{\circ})}$, obtaining a state feedback form directly from \eqref{optc} is challenging. To address this issue, we employ a change of probability measure by applying the Girsanov theorem under certain conditions, as stated in the following theorem.
 \begin{lemma}\label{thm:changeM} Consider the control action
\begin{equation} \label{feedbackcontrol}
u^{*}(t) =-R^{-1}\left [S^{\intercal}x_{t}-\zeta + B^{\intercal}\left (\Pi (t)x_{t}+s(t)  \right ) \right ]
\end{equation}
where the deterministic coefficients $\Pi(t)$ and $s(t)$ satisfy the ODEs 
\begin{align}   
&\dot{\Pi}(t)+\Pi(t)A+A^{\intercal}\Pi(t)-(\Pi(t)B+S)R^{-1}(B^{\intercal}\Pi(t)+S^{\intercal})\nonumber\\&\hspace{7cm}+Q+\delta \Pi(t)\sigma (t)\sigma^{\intercal} (t)\Pi(t)=0,\quad \Pi (T)=\widehat{Q}, \label{pi}\\
 &\dot{s}(t) +\left [ A^{\intercal}-\Pi(t)BR^{-1}B^{\intercal}-SR^{-1}B^{\intercal}+\delta\Pi(t)\sigma (t)\sigma^{\intercal} (t)  \right ]s(t)\nonumber\\&\hspace{5cm}+\Pi(t)(b(t)+BR^{-1}\zeta )+SR^{-1}\zeta -\eta =0,\quad s(T)=0. \label{s}
\end{align}
Moreover, let $C^{*}_T$ be defined by  
\begin{align}
C^{*}_T=&\frac{\delta }{2}\int_{0}^{T}\left(2\left<s(t),b(t)\right> -\left<R^{-1}(B^{\intercal}s(t)-\zeta ),B^{\intercal}s(t)-\zeta\right> \notag+\tr(\Pi (t)\sigma (t)\sigma^{\intercal} (t))\right)dt\\&+\frac{\delta^{2} }{2} \int_{0}^{T}\left\|\sigma^{\intercal} (t)s(t) \right\|^{2}dt +\frac{\delta }{2}\left<\Pi (0)x_{0},x_{0} \right> +\delta \left<s(0),x_{0}\right>.  \label{C_star} 
\end{align}
Then, the random variable $\exp(\delta \Lambda_{T}(u^{*})-C^{*}_T)$, where $\Lambda_{T}(u^{*})$ is given by \eqref{costG} under the control \eqref{feedbackcontrol}, is a Radon-Nikodym derivative 
\begin{align} 
\frac{d\widehat{\mathbb{P}}}{d\mathbb{P}}=\exp\left(\delta \Lambda_{T}(u^{*})-{C}^{*}(T)\right)
\end{align}
defining a probability measure $\widehat{\mathbb{P}}$ equivalent to $\mathbb{P}$. Further, the Radon-Nikodym derivative can be represented by 
\begin{gather}
    \frac{d\widehat{\mathbb{P}}}{d\mathbb{P}}:=\exp\left(-\frac{\delta ^{2}}{2}\int_{0}^{T}\left\|\gamma_t \right\|^{2}dt+\delta \int_{0}^{T} \gamma^\intercal_tdw_{t}\right)\label{rnd}\\
\gamma_t=\sigma ^{\intercal}(\Pi (t)x_{t}+s(t)). \label{gamma} 
\end{gather}
\end{lemma}

\begin{proof}
We first define the following quadratic functional of the state process
\begin{equation}
\frac{\delta }{2}\left< \Pi (t)x_{t},x_{t}\right> +\delta \left< s(t),x_{t}\right>\label{quad_func}
\end{equation}
with the deterministic coefficients $\Pi(t)$ and $s(t)$, motivated by \citet{duncan2013linear}. We then apply Itô's lemma to \eqref{quad_func} and integrate both sides from $0$ to $T$ to get 
\begin{align}\label{Ito}
&\frac{\delta }{2}\left< \Pi (T)x_{T},x_{T}\right> +\delta \left<s(T),x_{T}\right> -\frac{\delta }{2}\left< \Pi (0)x_0,x_0\right> -\delta \left<s(0),x_0 \right>\notag\\  &=\frac{\delta }{2}\int_{0}^{T}\Big(\left< \dot{\Pi}(t) x_t,x_t\right> +\left<(\Pi(t) A+ A^{\intercal}\Pi(t) )x_t,x_t\right> +2\left<B^{\intercal}\Pi(t) x_t,u_t \right> +2\left<\Pi(t) x_t,b(t)\right> +2\left<\dot{s}(t),x_t)\right>\notag\\  &+2\left< s(t),Ax_t+Bu_t+b(t)\right> +\tr(\Pi(t) \sigma(t) \sigma(t)^{\intercal} )\Big)dt+\frac{\delta }{2}\int_{0}^{T}2\sigma(t)^{\intercal}(\Pi(t) x_t+s(t)) dw_{t}.  
\end{align} 
Next, we add $\delta \Lambda_{T}(u)$ to both sides of \eqref{Ito}, yielding
\begin{align} \label{LambdaTmanip}
&\delta \Lambda_{T}(u) -\frac{\delta }{2}\left< \Pi (0)x_0,x_0\right> -\delta \left<s(0),x_0 \right>\notag\\&=\frac{\delta }{2}\int_{0}^{T}\Big(\left<QX_{t},X_{t} \right> +2\left<Su_{t},x_{t} \right> +\left<Ru_t, u_t \right>\notag\\&-2\left<\eta ,x_{t} \right>-2\left<\zeta ,u_{t}\right>\Big)dt+\frac{\delta }{2}\left< \widehat{Q}X_{T},X_{T}\right>
 -\frac{\delta }{2}\left< \Pi (T)x_{T},x_{T}\right>\notag\\& -\delta \left<s(T),x_{T} \right>+\frac{\delta }{2}\int_{0}^{T}\Big(\left< \dot{\Pi }(t)x_{t},x_{t}\right> +\left<(\Pi (t)A+ A^{\intercal}\Pi (t))x_{t},x_{t}\right> \notag\\&+2\left<B^{\intercal}\Pi (t)x_{t},u_{t} \right> +2\left<\Pi (t)x_{t},b(t) \right> +2\left<\dot{s}(t),x_{t}\right>\notag\\&+2\left< s(t),Ax_{t}+Bu_{t}+b(t)\right> +\tr(\Pi (t)\sigma(t) \sigma(t)^{\intercal} )\Big)dt\notag\\&+\frac{\delta }{2}\int_{0}^{T}2\sigma(t)^{\intercal}(\Pi(t)x_{t}+s(t)) dw_{t}.
\end{align}
Subsequently, we substitute \eqref{feedbackcontrol} in \eqref{LambdaTmanip} and reorganize the terms to represent $\delta\Lambda_{T}(u)$ as
\begin{align}\label{rep_initial}
\delta \Lambda_{T}(u)= -\frac{\delta ^{2}}{2} \int_{0}^{T}\left\|\gamma_t \right\|^{2}dt+ \delta \int_{0}^{T} \gamma^\intercal_t dw_{t} + {C}^{*}_T+\frac{1}{2}\Phi_T
\end{align}
where $\gamma_t$ and $C^{*}_T$ are, respectively, given by \eqref{gamma} and \eqref{C_star}, and 
$\Phi_T$ is a random variable defined by 
\begin{align}
\Phi_T&=\delta \int_{0}^{T}\Big<\Big(\dot{\Pi}(t)+Q+A^{\intercal}\Pi (t)+\Pi (t)A-(\Pi(t)B+S)R^{-1}(B^{\intercal}\Pi(t)+S^{\intercal})\notag\\&+\delta\Pi(t) \sigma (t)\sigma ^{\intercal}(t)\Pi(t)\Big)x_{t},x_{t}\Big>+\Big<\dot{s}(t)+( A^{\intercal}-\Pi(t)BR^{-1}B^{\intercal}-SR^{-1}B^{\intercal}+\delta\Pi(t)\sigma (t)\sigma^{\intercal}(t))s(t)\notag\\&+\Pi(t)(b(t)+BR^{-1}\zeta ) +SR^{-1}\zeta -\eta,x_{t} \Big>dt+\delta  \left<\widehat{Q}x_{T},x_{T} \right>-\delta \left<\Pi (T)x_{T},x_{T} \right> -2 \delta \left<s(T),x_{T}\right>.
\end{align}
From \eqref{rep_initial}, we observe that $\exp(\delta \Lambda_{T}(u)-C^{*}_T)$ is a Radon-Nikodym derivative if
\begin{equation}
\Phi_T=0, \qquad \mb{P}-a.s.
\end{equation}
The above condition is fulfilled if $\Pi(t)$ and $s(t)$, respectively, satisfy \eqref{pi} and \eqref{s}, and if \eqref{feedbackcontrol} holds.
Note that, after applying \eqref{feedbackcontrol}, the dynamics \eqref{dynamic} under the $\widehat{\mathbb{P}}$ measure become
\begin{equation} \label{dynamichat}
dx_{t}=\left (\widehat{A}(t)x_{t}+\widehat{b}(t)\right )dt+\sigma(t)d\widehat{w}_{t}
\end{equation}
where $\widehat{A}(t):=A-BR^{-1}(S^{\intercal}+B^{\intercal}\Pi(t))+\delta \sigma(t)\sigma^{\intercal}(t)\Pi(t)$ and $\widehat{b}(t):=-BR^{-1}(B^{\intercal}s(t)-\zeta)+b(t)+\delta \sigma(t)\sigma^{\intercal}(t)s(t)$. The strong solution of the linear SDE \eqref{dynamichat} is given by \begin{equation}  \label{statex}
x_{t}=\Upsilon_{\widehat{A}}(t) x_{0}+\Upsilon_{\widehat{A}}(t) \int_{0}^{t}\Upsilon_{\widehat{A}}^{-1}(s) \widehat{b}(t)ds+\Upsilon_{\widehat{A}}(t) \int_{0}^{t}\Upsilon_{\widehat{A}}^{-1}(s)\sigma (s)dw_{s}    
\end{equation}
where $\Upsilon_{\widehat{A}}(t)$ and $\Upsilon_{\widehat{A}}^{-1}(t)$ are, respectively, the solutions of the following matrix-valued ODEs 
\begin{gather} 
 \dot{\Upsilon}_{\widehat{A}}(t)=\widehat{A}(t)\Upsilon_{\widehat{A}}(t)dt, \quad\Upsilon_{\widehat{A}}(0)=I \notag \\
 \dot{\Upsilon}_{\widehat{A}}^{-1}(t)=-\Upsilon_{\widehat{A}}^{-1}(t)\widehat{A}(t)dt, \quad \Upsilon_{\widehat{A}}^{-1}(0)=I,\label{ODE_trans}
\end{gather}
and satisfy $\Upsilon_{\widehat{A}}(t)\Upsilon_{\widehat{A}}^{-1}(t)=I$ (\citet{yong1999stochastic}). Since $\Upsilon_{\widehat{A}}(t)$ and $\Upsilon_{\widehat{A}}^{-1}(t)$ are continuous and of bounded variation, the stochastic integral in \eqref{statex} is almost surely a pathwise Riemann–Stieltjes integral (\citet{chung1990introduction}). Therefore, the linear structure $\gamma_t=\sigma ^{\intercal}(\Pi (t)x_{t}+s(t))$ and \eqref{statex} ensure that 
\begin{equation}
 \exp(\delta \Lambda_{T}(u^{*})-C^{*}_T)=\exp(-\frac{\delta ^{2}}{2}\int_{0}^{T}\left\|\gamma_t \right\|^{2}dt+\delta \int_{0}^{T} \gamma^\intercal_t dw_{t}) \notag  
\end{equation}
 is a Radon-Nikodym derivative by Beneš's condition (see page 200, \citet{karatzas1991brownian}). Hence, $\frac{d\widehat{\mathbb{P}}}{d\mathbb{P}}:=\exp(\delta \Lambda_{T}(u^{*})-C^{*}_T)$ defines the probability measure $\widehat{\mathbb{P}}$ equivalent to $\mathbb{P}$.
\end{proof}
From Lemma \ref{thm:changeM}, under the control action $u^{*}$, the process
\begin{equation}
\widehat{M}_{t}(u^{*}):=\frac{M_{2,t}(u^{*})}{M_{1,t}(u^{*})}    
\end{equation}
where $M_{1,t}(u^{*})$ and $M_{2,t}(u^{*})$ are, respectively, given by \eqref{M1t} and \eqref{M2t}, is a $\widehat{\mathbb{P}}$-martingale represented by
\begin{equation} \label{martrs}
\widehat{M}_{t}(u^{*})=\widehat{\mathbb{E}}\left[ e^{A^{\intercal}T}\widehat{Q}x_{T}+\int_{0}^{T}e^{A^{\intercal}s}(Qx_{s}+Su^{*}_s-\eta)ds \Big| \mathcal{F}_{t} \right],
\end{equation}
which is obtained through the corresponding change of measure. 
\begin{theorem}\label{thm:cntrl_NewM} 
Let $u^{*}_t$, $\Pi(t)$ and $s(t)$ be defined as in Lemma \ref{thm:changeM}, then 
\begin{equation} \label{pt}
\Pi(t)x_{t}+s(t)=e^{-A^{\intercal}t}\widehat{M}_{t}(u^{*})-\int_{0}^{t}e^{A^{\intercal}(s-t)}(Qx_{s}+Su^{*}_s-\eta)ds  , \quad \widehat{\mathbb{P}}-a.s.
\end{equation}
\end{theorem}

\begin{proof}
By inspection, \eqref{pt} holds at the terminal time $T$. More specifically, substituting $\widehat{M}_{T}(u^{*})$ from \eqref{martrs} in the right-hand side of \eqref{pt} results in 
\begin{equation}
\widehat{Q}x_{T} = e^{-A^{\intercal}T}\widehat{M}_{T}(u^*)-\int_{0}^{T}e^{A^{\intercal}(s-T)}(Qx_{s}+Su_{s}^*-\eta)ds, \quad \widehat{\mathbb{P}}-a.s.
\end{equation}  
which is $\widehat{\mathbb{P}}-a.s.$ equal to $\Pi(T)x_{t}+s(T)$ according to \eqref{pi} and \eqref{s}. Hence, in order to establish the validity of \eqref{pt} for all $t \in \mc{T}$, it suffices to demonstrate that the infinitesimal variations of both sides of the equation are $\widehat{\mathbb{P}}$-almost surely equal.

By the martingale representation theorem, $\widehat{M}_{t}(u)$ may be expressed as  
\begin{equation} \label{mart}
\widehat{M}_{t}(u^*)=\widehat{M_{0}} +\int_{0}^{t}Z(s)d\widehat{w}_s,\quad \widehat{\mathbb{P}}-a.s.
\end{equation}
We apply Ito's lemma to both sides of \eqref{pt} and 
substitute \eqref{dynamic}, \eqref{feedbackcontrol}, \eqref{mart} as required. For the resulting drift and diffusion coefficients to be equal on both sides, the equations 
\begin{multline}
\Big(\dot{\Pi}(t)+Q+A^{\intercal}\Pi (t)+\Pi (t)A-(\Pi(t)B+S)R^{-1}(B^{\intercal}\Pi(t)+S^{\intercal})+\delta\Pi(t) \sigma (t)\sigma ^{\intercal}(t)\Pi(t)\Big)x_{t} \\ +\Big(\dot{s}(t)+( A^{\intercal}-\Pi(t)BR^{-1}B^{\intercal}-SR^{-1}B^{\intercal}+\delta\Pi(t)\sigma (t)\sigma^{\intercal}(t)  )s(t)\\+\Pi(t)(b(t)+BR^{-1}\zeta ) +SR^{-1}\zeta -\eta\Big)=0, \quad \widehat{\mathbb{P}}-a.s. \label{drift}
\end{multline}
and 
\begin{equation} 
\Pi (t)\sigma (t)=e^{-A^{\intercal}t}Z(t),  \quad \widehat{\mathbb{P}}-a.s. \label{vol}
\end{equation}
must hold for all $t\in \mc{T}$. It is evident that the requirement \eqref{drift} is met if \eqref{pi} and \eqref{s} hold. It remains to demonstrate that the requirement \eqref{vol} is subsequently met for all $t \in \mc{T}$. 

To determine $Z(t)$, we substitute \eqref{feedbackcontrol} and \eqref{statex} in \eqref{martrs} and equate the stochastic components of the resulting equation with those of \eqref{mart}. This leads to 
\begin{equation}
\int_{0}^{t}Z(s)d\widehat{w}_s=\widehat{\mathbb{E}}\left[ \underbrace{e^{A^{\intercal}T}\widehat{Q}\Upsilon_{\widehat{A}}(T) \int_{0}^{T}\Upsilon_{\widehat{A}}^{-1}(s)\sigma (s)dw_{s}}_{D_1}+\underbrace{\int_{0}^{T}e^{A^{\intercal}s}\widetilde{Q}(s)\Upsilon_{\widehat{A}}(s) \int_{0}^{s}\Upsilon_{\widehat{A}}^{-1}(r)\sigma (r)dw_{r}ds}_{D_2} \Big| \mathcal{F}_{t} \right] \label{eq_Z_initial}  \end{equation}
where $\widetilde{Q}(t)=Q-SR^{-1}(S^{\intercal}+B^{\intercal}\Pi(t))$. Further, let $\mathfrak{L}_{1}(t)=e^{A^{\intercal}t}\widetilde{Q}(t)\Upsilon_{\widehat{A}}(t)$, $\mathfrak{L}_{2}(t)=\Upsilon_{\widehat{A}}^{-1}(t)\sigma (t)$, and rewrite $\widehat{\mathbb{E}}\left[D_2|\mathcal{F}_{t}\right]$ in \eqref{eq_Z_initial} as 
\begin{align}
\widehat{\mathbb{E}}\left[D_2|\mathcal{F}_{t}\right] &= \widehat{\mathbb{E}}\Bigg[\int_{0}^{T}\mathfrak{L}_{1}(s)\int_{0}^{s}\mathfrak{L}_{2}(r)dw_{r}ds  \Big | \mathcal{F}_t\Bigg] \notag \allowdisplaybreaks\\
&=\int_{0}^{T}\mathfrak{L}_{1}(s)\widehat{\mathbb{E}}\left[\int_{0}^{s}\mathfrak{L}_{2}(r) d\widehat{w}_{r} \Big | \mathcal{F}_t \right]ds \notag \\ &=\int_{0}^{t}\mathfrak{L}_{1}(s)\widehat{\mathbb{E}}\left [\int_{0}^{s}\mathfrak{L}_{2}(r) d\widehat{w}_{r} \Big | \mathcal{F}_t \right]ds+\int_{t}^{T}\mathfrak{L}_{1}(s)\widehat{\mathbb{E}}\left [\int_{0}^{s}\mathfrak{L}_{2}(r) d\widehat{w}_{r} \Big | \mathcal{F}_t \right]ds \notag\allowdisplaybreaks\\ &=\int_{0}^{t}\mathfrak{L}_{1}(s)\left [\int_{0}^{s}\mathfrak{L}_{2}(r) d\widehat{w}_{r} \right]ds+\int_{t}^{T}\mathfrak{L}_{1}(s)\left [\int_{0}^{t}\mathfrak{L}_{2}(r) d\widehat{w}_{r} \right] ds 
\notag \allowdisplaybreaks\\ & =\int_{0}^{t}\int_{r}^{t}\mathfrak{L}_{1}(s)ds\mathfrak{L}_{2}(r)d\widehat{w}_{r}+\int_{0}^{t}\int_{t}^{T}\mathfrak{L}_{1}(s)ds\mathfrak{L}_{2}(r)d\widehat{w}_{r} \qquad\quad\text{\scriptsize (change order of integration)}
\notag \allowdisplaybreaks\\ & = \int_{0}^{t}\int_{r}^{T}\mathfrak{L}_{1}(s)ds\mathfrak{L}_{2}(r)d\widehat{w}_{r} 
\end{align}
where the fourth equality holds due to the measurability and martingale property of $\int_{0}^{s}\mathfrak{L}_{2}(r) d\widehat{w}_{r}$. Moreover, due to the martingale property, we can rewrite $\widehat{\mathbb{E}}\left[D_1|\mathcal{F}_{t}\right] $ in \eqref{eq_Z_initial} as  
\begin{equation} \widehat{\mathbb{E}}\left[D_1|\mathcal{F}_{t}\right] = \widehat{\mathbb{E}}\left[ e^{A^{\intercal}T}\widehat{Q}\Upsilon_{\widehat{A}}(T) \int_{0}^{T}\Upsilon_{\widehat{A}}^{-1}(s)\sigma (s)dw_{s} \Big | \mathcal{F}_{t} \right]=
e^{A^{\intercal}T}\widehat{Q}\Upsilon_{\widehat{A}}(T)  \int_{0}^{t}\mathfrak{L}_{2}(s)dw_{s}. \label{D2_exp}
 \end{equation}
Since the martingale representation theorem ensures the uniqueness of the expression $Z(t)$, from \eqref{eq_Z_initial}-\eqref{D2_exp}, we conclude that 
\begin{equation}\label{Z}
 Z(t)= \int_{t}^{T}\mathfrak{L}_{1}(r)dr\mathfrak{L}_{2}(t)+e^{A^{\intercal}T}\widehat{Q}\Upsilon_{\widehat{A}}(T)\mathfrak{L}_{2}(t),\quad \forall t\in \mc{T}.
\end{equation}
We proceed by using the representation
\begin{equation}
 e^{-A^{\intercal}t}Z(t)=\widetilde{Z}(t)\sigma(t),
\end{equation}
where 
\begin{equation}
\widetilde{Z}(t):=e^{-A^{\intercal}t}(\int_{t}^{T}\mathfrak{L}_{1}(s)ds\Upsilon_{\widehat{A}}^{-1}(t)+e^{A^{\intercal}T}\widehat{Q}\Upsilon_{\widehat{A}}(T)\Upsilon_{\widehat{A}}^{-1}(t))=\Pi(t),\quad \widehat{\mathbb{P}}-a.s. \label{Z_tilde}
\end{equation}
for all $t\in \mc{T}$. We will now demonstrate that $\widetilde{Z}(t) = \Pi(t)$, which verifies equation \eqref{vol}.
Since $\widetilde{Z}(T)=\widehat{Q}=\Pi(T)$, it is enough to show that $\widetilde{Z}(t)$ and $\Pi(t)$ satisfy the same ODE. From \eqref{Z_tilde}, we have 
\begin{align} \label{ztil}
 \dot{\widetilde{Z}}(t)&=-A^{\intercal}\widetilde{Z}(t)+e^{-A^{\intercal}t}\left (-\mathfrak{L}_{1}(t)\Upsilon_{\widehat{A}}^{-1}(t)-\int_{t}^{T}\mathfrak{L}_{1}(s)ds\Upsilon_{\widehat{A}}^{-1}(t)\widehat{A}(t)-e^{A^{\intercal}T}\widehat{Q}\Upsilon_{\widehat{A}}(T)\Upsilon_{\widehat{A}}^{-1}(t)\widehat{A}(t)  \right ) \notag \\ & =-A^{\intercal}\widetilde{Z}(t)-\widetilde{Q}(t)-\widetilde{Z}(t)\widehat{A}(t)\notag \\ &=-A^{\intercal}\widetilde{Z}(t)-Q+SR^{-1}(S^{\intercal}+B^{\intercal}\Pi(t))-\widetilde{Z}(t)(A-BR^{-1}(S^{\intercal}+B^{\intercal}\Pi(t))-\delta \sigma(t)\sigma^{\intercal}(t)\Pi(t))\notag \\ &=-A^{\intercal}\widetilde{Z}(t)-\widetilde{Z}(t)A-Q+(\widetilde{Z}(t)B+S)R^{-1}(B^{\intercal}\Pi(t)+S^{\intercal})-\delta \widetilde{Z}(t)\sigma (t)\sigma^{\intercal} (t)\Pi(t),
\end{align}
which is a first-order linear ODE of the Sylvester type with $\Pi(t)$ fixed as the solution of \eqref{pi}. This ODE admits a unique solution that coincides with $\Pi(t)$ (see for example \citet{behr2019solution}). This completes the proof. 
\end{proof}
The following corollary demonstrates that the control action $u^{*}$ given by \eqref{feedbackcontrol}-\eqref{s} is indeed the optimal control action for the LQG risk-sensitive system under $\mathbb{P}$.
\begin{corollary} \label{coro4}Under Assumption 1, $u^*$ given by \eqref{feedbackcontrol}-\eqref{s} is the optimal control action of the LQG risk-sensitive system governed by \eqref{dynamic}-\eqref{costG}.  
\end{corollary}
\begin{proof}
It is enough to show that the control action $u^{*}$ given by \eqref{feedbackcontrol}-\eqref{s} satisfies the necessary and sufficient optimality condition given by \eqref{optc}.
We have 
\begin{align}
u^{*}_t& =-R^{-1}\left [S^{\intercal}x_{t}-\zeta + B^{\intercal}\left (\Pi (t)x_{t}+s(t)  \right ) \right ]\notag \\ &=-R^{-1}\left [S^{\intercal}x_{t}-\zeta + B^{\intercal}\left (e^{-A^{\intercal}t}\widehat{M}_{t}(u^{*})-\int_{0}^{t}e^{A^{\intercal}(s-t)}(Qx_{s}+Su^{*}_s-\eta)ds  \right ) \right ]\quad \text{{\small(by Theorem \ref{thm:cntrl_NewM})}}\notag \\ &=-R^{-1}\left [S^{\intercal}x_{t}-\zeta + B^{\intercal}\left (e^{-A^{\intercal}t}\frac{M_{2,t}(u^{*})}{M_{1,t}(u^{*})}  -\int_{0}^{t}e^{A^{\intercal}(s-t)}(Qx_{s}+Su^{*}_s-\eta)ds  \right ) \right ],\,\,\,  \widehat{\mathbb{P}}-a.s., \end{align}
where the last equality is a direct result of Lemma \ref{thm:changeM}. Finally, due to the equivalence of $\widehat{\mathbb{P}}$ and $\mathbb{P}$, we have 
\begin{equation}  
u^{*}_t =-R^{-1}\Bigg[S^{\intercal}x_{t}-\zeta+B^{\intercal}\left (e^{-A^{\intercal}t}\frac{M_{2,t}(u^{*})}{M_{1,t}(u^{*})}-\int_{0}^{t}e^{A^{\intercal}(s-t)}(Qx_{s}+Su^{*}_s-\eta)ds  \right )\Bigg],\,\, \mathbb{P}-a.s.    
\end{equation}
Hence, $u^\ast$ is an optimal control action for the system described by \eqref{dynamic}-\eqref{costG}. The uniqueness of $u^\ast$ as the optimal control action is established due to the strict convexity of the cost functional \eqref{cost}-\eqref{costG}, as demonstrated in \Cref{strict_convexity}.
\end{proof}
\begin{remark}(Risk-Neutral Probability Measure) The probability measure $\widehat{\mathbb{P}}$ may be termed risk-neutral because, under this measure, the necessary and sufficient optimality condition for risk-sensitive LQG optimal control problems described by \eqref{dynamic}-\eqref{costG} is expressed as 
\begin{gather}
\label{optc-new-measure}
u^{*}_t =-R^{-1}\Bigg[S^{\intercal}x_{t}-\zeta+B^{\intercal}\left(e^{-A^{\intercal}t}\widehat{M}_{t}(u^{*})-\int_{0}^{t}e^{A^{\intercal}(s-t)}(Qx_{s}+Su^{*}_s-\eta)ds  \right)\Bigg],
\end{gather}
with $\widehat{M}_{t}$ given by \eqref{martrs}, which is similar to the optimality condition of risk-neutral LQG optimal control problems as detailed in \cite[eq (24)]{firoozi2020convex}. More specifically, under this measure, the optimality condition of the risk-sensitive optimal control problems described by \eqref{dynamic}-\eqref{costG} has the same structure as that of risk-neutral optimal control problems described by
\begin{gather}
dx_{t}=(Ax_{t}+Bu_{t}+b(t))dt+\sigma dw_{t},\\
J(u)=\mathbb{E}\left[ \Lambda_{T}(u) \right], 
\end{gather}
where $\Lambda_{T}(u)$ is given by \eqref{costG}.
\end{remark}
The results obtained in this section can be readily extended to the case where the system matrix is time varying. The following remark provides a summary of this extension, which will be used in \Cref{sec_game}. 
\begin{remark}(Time-varying system matrix $A(t)$)
Consider the system described by the dynamics
\begin{equation}
dx_{t}=(A(t)x_{t}+Bu_{t}+b(t))dt+\sigma (t)dw_{t}, \label{TV_case}
\end{equation}
where $A(t)$ is a continuous function on $\mc{T}$, and the cost functional is given by \eqref{cost}-\eqref{costG}. With some slight modifications, the G\^ateaux derivative of the cost functional is given by
     \begin{align}
    \left<\mathcal{D}J(u),\omega \right> = \delta \mathbb{E}\Bigg[ \int_{0}^{T}\omega ^{\intercal}[B^{\intercal}(\Upsilon ^{-1}_{A}(t))^{\intercal} &M_{2,t}(u)+M_{1,t}(u)(Ru_{t}+S^{\intercal}x_{t}-\zeta \\ &+B^{\intercal}(\Upsilon ^{-1}_{A}(t))^{\intercal} \int_{0}^{t}(\Upsilon _{A}(s))^{\intercal} (Qx_{s}+Su_{s}-\eta)ds )]dt\Bigg],  \notag
    \end{align}
    where $\Upsilon_{A}(t)$ and $\Upsilon_{A}^{-1}(t)$  are defined in the same way as in \eqref{ODE_trans}. The martingale term $M_{1,t}(u)$ is given by \eqref{M1t} and $ M_{2,t}(u)$ by 
    \begin{equation}
      M_{2,t}(u)= \mathbb{E}\left [ e^{\frac{\delta}{2 }\Lambda_{T}(u)}((\Upsilon _{A}(T))^{\intercal} \widehat{Q}x_{t}+\int_{0}^{T}(\Upsilon _{A}(s))^{\intercal} (Qx_{s}+Su_{s}-\eta)ds)\Big | \mathcal{F}_{t} \right ].   
    \end{equation}
    Subsequently, the necessary and sufficient optimality condition for the control action $u^{\circ}_t$ is given by
    \begin{equation}
    u^{\circ}_t =-R^{-1}\Bigg[S^{\intercal}x_{t}-\zeta+B^{\intercal}\left ((\Upsilon ^{-1}_{A}(t))^{\intercal}\frac{M_{2,t}(u^{\circ})}{M_{1,t}(u^{\circ})}-(\Upsilon ^{-1}_{A}(t))^{\intercal} \int_{0}^{t}(\Upsilon _{A}(s))^{\intercal}(Qx_{s}+Su^{\circ}_s-\eta)ds\right ) \Bigg]. 
    \end{equation}
    By applying adapted versions of \cref{thm:changeM}, \cref{thm:cntrl_NewM}, and \cref{coro4}, it can be shown that the optimal control action is given by \eqref{feedbackcontrol}-\eqref{s}, where $A$ is replaced with $A(t)$ in \eqref{pi}-\eqref{s}. 
    
\end{remark}
\begin{remark}[Technical Comparison with Existing Methodologies]\label{comparison}
\!\!The work \cite{lim2005new} develops a risk-sensitive maximum principle requiring that the running and terminal cost functionals in the exponent be uniformly bounded and Lipschitz continuous (see \cite[Assumption B2]{lim2005new}.). This condition is not automatically met for LQG risk-sensitive models with quadratic running and terminal costs in the exponent, unless the state and control spaces are restricted to compact sets. In \cite{duncan2013linear}, a combination of completing the square and a Radon-Nikodym derivative is used to determine an optimal control for an LQG risk-sensitive problem. To verify the optimality of a candidate control, a perturbation process (see \cite[eq. (14)]{duncan2013linear}) is introduced over a specific subset of the time interval. This perturbation of the control action is a bounded process, although the admissible control set includes $L^2$ processes. The work \cite{jacobson1973optimal}  uses dynamic programming to obtain solutions to continuous-time risk-sensitive optimal control problems, where no verification theorem is presented. Similarly, \cite{fleming2006controlled} employs a dynamic programming approach to address such risk-sensitive problems. However, the verification theorem in this work assumes that both the state and control spaces are  bounded (See \cite[Thm. 8.2 \& eq. (3.12)]{fleming2006controlled}). The works \cite{moon2018risk}, \cite{bacsar1998dynamic}, and \cite{moon2019risk}, which respectively study risk-sensitive two-player games and mean-field games, employ similar methodologies and impose conditions similar to those in \cite{lim2005new}. They require uniform boundedness and Lipschitz continuity for running and terminal costs in the exponent. Furthermore, they indicate that restricting state and control spaces to sufficiently large compact subsets of Euclidean spaces is necessary for applying the methodology to LQG counterpart models (see \cite[Sec. VII, footnote 8]{moon2018risk}, \cite[Chap. 6]{bacsar1998dynamic}, \cite[Example 1]{moon2019risk}.). Our variational approach takes advantage of the fact that the Gâteaux derivative may be computed explicitly for LQG risk-sensitive models. Given that the cost functional is strictly convex, this allows us to obtain the necessary and sufficient condition of optimality by setting the Gâteaux derivative to zero. However, the interchangeability of the limit and expectation is required to enable the calculation of the Gâteaux derivative (see \eqref{deriv-0}-\eqref{deri1} in the proof of \Cref{thm:Cntrl_initial}). 
    Although boundedness of state and control processes, as assumed in the literature, provides a sufficient condition for this interchangeability, it is not a necessary condition. 
\end{remark}
The variational analysis developed above will be employed in the next section to obtain the best-response strategies of major and minor agents in MFG systems.
\section{Major-Minor LQG Risk-Sensitive Mean-Field Game Systems}
\label{sec_game}
 Risk-neutral MFGs including a major agent and a large number of minor agents were first introduced in \cite{huang2010large} and have since attracted significant research interest \cite{nourian2013epsilon,carmona2016probabilistic,carmona2017alternative,lasry2018mean,cardaliaguet2019master,firoozi2020convex,huang2020linear,firoozi2022lqg}. In this section, we focus on Risk-Sensitive LQG MFGs that incorporate a major agent and many minor agents. We develop a variational analysis to tackle this problem. Our methodology builds on the variational analysis developed for risk-neutral MFGs as detailed in \cite{firoozi2020convex}.
\subsection{Finite-Population Model}\label{sec:finite_pop}
We consider a system that contains one major agent, who has a significant impact on other agents, and $N$ minor agents, who individually have an asymptotically negligible impact on the system. Minor agents form $K$ subpopulations, such that the agents in each subpopulation share the same model parameters. We define  the index set $\mc{I}_k = \lbrace i: \theta_i = \theta^{(k)} \rbrace,\, k \in \mc{K} :=\left\{1,\dots,K \right\}$, where $\theta^{(k)}$ denotes the model parameters of subpopulation $k$ that will be introduced throughout this section. 
Moreover, we denote the empirical distribution of the parameters $(\theta^{(1)},\dots,\theta^{(K)})$ by $\pi^{(N)}=(\pi_1^{(N)}....\pi_K^{(N)})$, where $\pi_k^{(N)}=\frac{\left| \mc{I}_k\right|}{N}$ and $\left| \mc{I}_k\right|$ is the counting measure of $\mc{I}_k$.

The dynamics of the major agent and of a representative minor agent indexed by $i$ in subpopulation $k$ are, respectively, given by
\begin{align} 
&dx^{0}_{t}=(A_{0}x^{0}_{t}+F_{0}x^{(N)}_{t}+B_{0}u^{0}_t+b_{0}(t))dt+\sigma _{0}(t)dw^{0}_{t} \label{dynamicmaj}\\
&dx^{i}_{t}=(A_{k}x^{i}_{t}+F_{k}x^{(N)}_{t}+G_{k}x^{0}_{t}+B_{k}u^{i}_{t}+b_{k}(t))dt+\sigma _{k}(t)dw^{i}_{t} \label{dynamicmin}
\end{align}
where $i \in \mc{N} =\left\{ 1,\dots,N\right\}$, $k\in \mc{K}$, and $t\in \mc{T}$. The state and the control action are denoted, respectively, by $x^{i}_{t}\in \mathbb{R}^{n}$ and $u^{i}_{t} \in \mathbb{R}^{m}$, $i \in \mc{N}_0=\{ 0,1,\dots,N\}$. Moreover, the processes $\{w^{i}\in \mathbb{R}^{r},i\in \mc{N}_0\}$, are ($N+1$) standard $r$-dimensional Wiener processes defined on the filtered probability space $\left ( \Omega,\mathcal{F},\{\mathcal{F}_{t}^{(N)}\}_{t\in \mc{T}},\mathbb{P} \right )$, where $\mathcal{F}_{t}^{(N)}:=\sigma (x^{i}_0, w^{i}_{s}, i \in \mc{N}_0, s\leq t)$. Finally, the volatility processes $\sigma _{0}(t),\sigma _{k}(t)\in \mathbb{R}^{n\times r}$ and the offset processes $b_{0}(t),b_{k}(t)\in \mathbb{R}^{n}$ are deterministic functions of time, while all other parameters $ \left (A_0,F_0,B_0\right )$, $ \left (A_k,F_k,B_k  \right )$ are constants of an appropriate dimension. 

The empirical average state $x^{(N)}_{t}$ of minor agents is defined by
\begin{equation}
 x^{(N)}_{t}:=\frac{1}{N} \sum_{i \in \mc{N}}x^{i}_{t}   
\end{equation}
where the same weight is assigned to each minor agent's state, implying that minor agents have a uniform impact on the system. 
Denoting $u^{-0} \coloneqq  (u^1,\dots,u^N)$ and $u^{-i} \coloneqq  (u^0,\dots,u^{i-1}, u^{i+1},\dots, u^N)$, the major agent's cost functional is given by 
\begin{align} \label{costmaj}
J_{0}^{(N)}(u^{0},u^{-0})=&\mathbb{E} \bigg[\exp\bigg(\frac{\delta_{0}}{2}\left<\widehat{Q}_{0}(x^{0}_{T}-\Phi^{(N)}_{T}),x ^{0}_{T}-\Phi^{(N)}_{T}\right> +\frac{\delta_{0}}{2}\int_{0}^{T}\Big(\left< Q_{0}(x^{0}_{t}-\Phi ^{(N)}_{t}),x^{0}_{t}-\Phi ^{(N)}_{t}\right>\notag \\  &+2\left< S_{0}u^{0}_{t},x^{0}_{t}-\Phi^{(N)}_{t}\right>+\left< R_0u^{0}_{t},u^{0}_{t}\right> \Big)dt\bigg)\bigg],
\end{align}  
where $\Phi ^{(N)}_{t}:=H_{0}x^{(N)}_{t}+ \eta_0$ and all parameters $(\widehat{Q}_{0}$, $Q_{0}$, $S_{0}$, $H_0$, $\widehat{H}_0$, $R_0$, $\eta_0)$ are of an appropriate dimension.
\begin{assumption}\label{assum:majorCost}
$R_{0}>0$, $\widehat{Q}_{0}\geq 0$,  $Q_{0}-S_{0}R^{-1}_{0}S_{0}^{\intercal } \geq 0$, and $\delta_{0} \in (0,\infty)$. 
\end{assumption}
For the representative minor agent $i$ in subpopulation $k$, the cost functional is given by
\begin{align} \label{costmin}
J_{i}^{(N)}(u^{i},u^{-i})=&\mathbb{E}\bigg[ \mathrm{exp}\bigg(\frac{\delta_{k}}{2}\left<\widehat{Q}_{k}(x ^{i}_{T}-\Psi^{(N)}_{T}),x ^{i}_{T}-\Psi^{(N)}_{T}\right> +\frac{\delta_{k}}{2}\int_{0}^{T}\Big(\left< Q_{k}(x^{i}_{t}-\Psi  ^{(N)}_{t}),x^{i}_{t}-\Psi  ^{(N)}_{t}\right>\notag \\ &  +2\left< S_{k}u^{i}_{t},x^{i}_{t}-\Psi^{(N)}_{T}\right> +\left< R_{k}u^{i}_{t},u^{i}_{t}\right> \Big) dt \bigg)\bigg] 
\end{align}
where $\Psi ^{(N)}_{t}:=H_{k}x^{0}_{t}+\widehat{H}_{k}x^{(N)}_{t}+\eta_{k}$ and all parameters $(\widehat{Q}_{k}$, $Q_{k}$, $S_{k}$, $H_k$, $\widehat{H}_k$, $R_k$, $\eta_k)$ are of an appropriate dimension.

\begin{assumption}\label{assum:minorCost}
$R_{k}>0$, $\widehat{Q}_{k}\geq 0$,  $Q_{k}-S_{k}R^{-1}_{k}S_{k}^{\intercal } \geq 0$, and $\delta_{k}\in (0,\infty)$, $\forall k\in \mc{K}$.
\end{assumption}
Under Assumptions \ref{assum:majorCost}\textendash\ref{assum:minorCost}, \eqref{costmaj} and \eqref{costmin} are strictly convex. 

From \eqref{dynamicmaj}-\eqref{costmin}, the dynamics and cost functionals of both the major agent and the representative minor agent-$i$ are influenced by the empirical average state $x^{(N)}_{t}$. Moreover, the representative minor agent's model is also influenced by the major agent's state $x^{0}_{t}$. 

For both the major agent and the representative minor agent-$i$, an admissible set $\mc{U}^g$ of control actions consists of all  $\mathbb{R}^{m}$-valued $\mathcal{F}_{t}^{(N)}$-adapted processes $u^{i}_{t},\, i\in \mc{N}_0$, such that $\mathbb{E}\left [ \int_{0}^{T} \left \| u^{i}_{t} \right \|^{2}dt\right ]< \infty$. 

In general, solving the $N$-player differential game described in this section becomes challenging, even for moderate values of $N$. The interactions between agents lead to a high-dimensional optimization problem, where each agent needs to observe the states of all other interacting agents. To address the dimensionality and the information restriction, we investigate the limiting problem as the number of agents $N$ tends to infinity. In this limiting model, the average behavior of the agents, known as the mean field, can be mathematically characterized, simplifying the problem. Specifically, in the limiting case, the major agent interacts with the mean field, while a representative minor agent interacts with both the major agent and the mean field. In the next sections, we derive a Markovian closed-loop Nash equilibrium for the limiting game model and show that it yields an $\epsilon$-Nash equilibrium for the original finite-player model. 
\subsection{Infinite-Population Model}\label{sec_inf_pop}
In order to derive the limiting model, we begin by imposing the following assumption.
\begin{assumption}
 There exists a vector of probabilities $\pi$ such that $\lim_{N \to \infty}\pi^{(N)}=\pi$.   
\end{assumption} 
\underline{Mean Field}: We first characterize the average state of minor agents in the limiting case. The average state of subpopulation $k$ is defined by
\begin{align} \label{subave}
x^{(N_k)}_{t} = \tfrac{1}{N_k} \sum_{i\in\mc{I}_k} x^{i}_{t}.
\end{align}
Let  $(x^{[N]})^{\intercal}= [(x^{(N_1)})^{\intercal}, (x^{(N_2)})^{\intercal},\dots, (x^{(N_K)})^{\intercal}]$. If it exists, the pointwise in time limit (in quadratic mean) of $x^{(N)}_{t}$ is called the \emph{state mean field} of the system and denoted by $\bar{x}^\intercal= [(\bar{x}^1)^\intercal, ..., (\bar{x}^K)^\intercal]$. Equivalently, in the limiting case, the representation $\bar{x}^k_t=\mb{E}[x^{\cdot k}|\mc{F}^0_t]$ may be used, where $x^{\cdot k}$ denotes the state of a representative agent in subpopulation $k$ (\citet{nourian2013epsilon,carmona2017alternative}). 

In a similar manner, we define the vector $(u^{[N]})^{\intercal}= [(u^{(N_1)})^{\intercal}, (u^{(N_2)})^{\intercal},\dots, (u^{(N_K)})^{\intercal}]$, the pointwise in time limit (in quadratic mean) of which, if it exists, is called the \emph{control mean field} of the system and denoted by $\bar{u}^\intercal= [(\bar{u}^1)^\intercal, ..., (\bar{u}^K)^\intercal]$. 
We can obtain the SDE satisfied by the state mean field $\bar{x}^k$ of subpopulation $k$ by taking the average of the solution $x^i_t$ to \eqref{dynamicmin} for all agents in subpopulation $k$ (i.e., $\forall i\in\mc{I}_k$), and then taking its $L^2$ limit as $N_k\to\infty$. This SDE is given by 
\begin{equation}
 d\bar{x}^{k}_{t} =
  \left[(A_k\,\mathbf{e}_k+F^{\pi}_k)\,\bar{x}_{t} + G_k \,x^{0}_{t} + B_k\, \bar{u}^{k}_{t} + b_k(t)
 \right] dt,
\end{equation}
where $F_k^{\pi}= \pi \otimes F_k \coloneqq \left[\pi_1 F_k, \dots, \pi_K F_k \right]$, and $\mathbf{e}_k = [0_{n \times n}, ..., 0_{n \times n},\Id_n, 0_{n \times n}, ..., 0_{n \times n}]$, where the $n \times n$ identity matrix $\Id_n$ appears in the $k$th block, and the $n \times n$ zero matrix appears in all other blocks. The dynamics of the mean-field vector $(\bar{x}_{t})^\intercal := [(\bar{x}^{1}_{t})^\intercal, \dots, (\bar{x}^{K}_{t})^\intercal]$, referred to as the \emph{mean-field equation}, are then given by
\begin{equation} \label{mfeq}
d\bar{x}_{t} = \left(\br{A}\,\bar{x}_{t} + \br{G}\, x^{0}_{t} + \br{B}\, \bar{u}_{t} + \br{m}(t)\right)dt,
\end{equation}
where
\begin{equation}\label{MFmatrices}
\br{A} = \begin{bmatrix}
A_1\mathbf{e}_1+ F^{\pi}_1\\
\vdots \\
A_K\mathbf{e}_K+F^{\pi}_K
\end{bmatrix},
\quad
\br{G} = \begin{bmatrix}
G_1 \\
\vdots\\
G_K
\end{bmatrix},
\quad
\br{B} = \begin{bmatrix}
B_1 & &0\\
        & \ddots &\\
       0 & & B_K
\end{bmatrix},
\quad
\br{m}(t) = \begin{bmatrix}
b_1(t)\\
\vdots\\
b_K(t)
\end{bmatrix}.
\end{equation}
\underline{Major Agent}: In the limiting case, the dynamics of the major agent are given by
\begin{equation} \label{majordy1}
dx^{0}_{t} = [A_0\, x^{0}_{t} + F_0^{\pi} \,\bar{x}_{t} + B_0u^{0}_{t}+ b_0(t)]dt + \sigma_{0}(t)dw^{0}_{t},
\end{equation}
where $F_0^{\pi}\coloneqq \pi \otimes F_0^{\pi} = \left[\pi_1 F_0, \dots, \pi_K F_0 \right]$ and the empirical state average is replaced by the state mean field. Following \citet{huang2010large}, in order to make the major agent's model Markovian, we form the extended state 
$(X^{0}_{t})^{\intercal}:=\left [(x^{0}_{t})^{\intercal},(\bar{x}_{t})^{\intercal}  \right ]$ satisfying 
\begin{equation}\label{majorExtDynPert0}
dX^{0}_{t} = \left(\widetilde{A}_0 \,X^{0}_{t} + \mb{B}_0\, u^0_t + \widetilde{B}_0\, \bar{u}_t+\widetilde{M}_{0}(t)\right )dt + \Sigma_{0} dW^{0}_{t},
\end{equation}
where
\begin{gather}
 \widetilde{A}_0 = \left[ \begin{array}{cc}
A_0 & F_0^{\pi} \\
\br{G} &  \br{A}
\end{array} \right]\!,
\quad \mb{B}_0=\left[ \begin{array}{c} B_0 \\ 0  \end{array}\right]
\!,
\quad
 \widetilde{B}_0 = \left[ \begin{array}{c} 0 \\ \br{B}  \end{array}\right]\!,\notag \\
\widetilde{M}_0(t) = \left[ \begin{array}{c} b_0(t) \\ \br{m}(t)  \end{array}\right]\!,
\quad
\Sigma_0 = \left[ \begin{array}{cc}
\sigma_0 & 0 \\
0 &  0
\end{array} \right]\!,
\quad W^{0}_{t} = \left[ \begin{array}{c}
w^{0}_{t}\\
0
\end{array} \right]\!.\label{sysMatMajor}
\end{gather}
The cost functional $J^{\infty }_{0}(\cdot)$ of the major agent under this framework is given by
\begin{align} 
J_{0}^{\infty}(u^{0})&=\mathbb{E}\Bigg[ \mathrm{exp}(\frac{\delta_{0}}{2}\left<\mathbb{G} _{0}X^{0}_{t},X^{0}_{t}\right> +\frac{\delta_{0}}{2}\int_{0}^{T}\left<\mathbb{Q} _{0}X_{s}^{0},X_{s}^{0}\right>+2\left<\mathbb{S}_0u^{0}_s,X_{s}^{0}\right>\notag\\&\hspace{5cm}+\left< R_{0}u^{0}_s,u^{0}_s\right>-2\left< X_{s}^{0},\bar{\eta }_{0}\right> -2\left<u^{0}_s,\bar{n}_0 \right>dt) \Bigg]\label{costinmaj} \\
\mathbb{G}_0 &= \left[\Id_{n}, -H_0^{\pi}\right ]^\intercal \hQ_0   \left [\Id_{n}, -H_0^{\pi}\right],\quad  \mathbb{Q}_0 = \left[\Id_{n}, -H_0^{\pi}\right ]^\intercal Q_0 \left[\Id_{n}, -H_0^{\pi}\right],\quad  
\mathbb{S}_0 = \left[\Id_{n}, -H_0^{\pi}\right ]^\intercal  S_0, \notag\\&\hspace{2cm} \bar{\eta}_0 = \left[\Id_{n}, -H_0^{\pi}\right ]^\intercal  Q_0 \eta_0, \quad \bar{n}_0 = S_0^\intercal \eta_0,\quad 
H_0^{\pi} = \left[\pi_1 H_0, \dots, \pi_K H_0\right]. \label{majorpa}
\end{align} 
\underline{Minor Agent}: The limiting dynamics of the representative minor agent $i$ in subpopulation $k$ are given by 
\begin{equation} \label{minordyinf1}
dx^{i}_{t} = [A_k\, x^{i}_{t} + F^{\pi}_k\, \bar{x}_{t} + G_k \,x^{0}_{t}  + B_k \,u^{i}_{t} + b_k(t)]dt + \sigma_k\, dw^i_t,
\end{equation}
where $F_k^{\pi}\coloneqq \pi \otimes F_k^{\pi} = \left[\pi_1 F_k, \dots, \pi_K F_k \right]$. As for the major agent, we form  the representative minor agent's extended state $(X^{i}_{t})^{\intercal}:=\left [(x^{i}_{t})^{\intercal},(x^{0}_{t})^{\intercal},(\bar{x}_{t})^{\intercal}  \right ]$ in order to make the model Markovian. The extended dynamics are given by
\begin{equation} \label{minorExtDynPert}
dX^{i}_t=(\widetilde{A}_{k}X^{i}_t+\mathbb{B}_{k}u^{i}_{t}+\widetilde{\mb{B}}_{0}u^{0}_{t}+\widetilde{B}\bar{u}_t+\widetilde{M}_{k}(t))dt+\Sigma_{k}dW^{i}_{t},     
\end{equation}
where \eqref{majorExtDynPert0} and \eqref{minordyinf1} are used, and
\begin{gather}
\widetilde{A}_k = \left[ \begin{array}{cc} A_k & [G_k \, \, \, F_k^{\pi}]\\ 0 &\widetilde{A}_0 \end{array} \right], \quad
 \mb{B}_k = \left[ \begin{array}{c} B_k \\ 0 \end{array}\right], \quad \widetilde{\mb{B}}_0 = \left[ \begin{array}{c} 0 \\ \mb{B}_0 \end{array}\right], \quad
\widetilde{B} = \left[ \begin{array}{c} 0 \\ \widetilde{B}_0 \end{array}\right]\notag\\
 \widetilde{M}_k(t) = \left[ \begin{array}{c} b_k(t) \\ \widetilde{M}_0(t) \end{array}\right], \quad
\Sigma_k = \left[ \begin{array}{cc} \sigma_k & 0 \\ 0 & \Sigma_0 \end{array} \right], \quad W^{i}_{t} = \left[ \begin{array}{c} w^{i}_{t} \\ W^{0}_{t} \end{array} \right].\label{sysMatMinor}
\end{gather}
The cost functional for minor agent $i$, expressed in terms of its extended state, can be reformulated as 
\begin{multline} \label{costmininf}
J_{i}^{\infty }(u^{i})=\mathbb{E}\Big[ \exp\Big(\frac{\delta_{k}}{2}\left<\mathbb{G} _{k}X^{i}_{t},X^{i}_{t}\right> +\frac{\delta_{k}}{2}\int_{0}^{T}\left<\mathbb{Q} _{k}X_{s}^{i},X_{s}^{i}\right>+2\left<\mathbb{S} _{k}u^{i}_s,X_{s}^{i}\right>\\+\left< R_{k}u^{i}_s,u^{i}_s\right>-2\left< X_{s}^{i},\bar{\eta }_{k}\right> -2\left<u^{i}_s,\bar{n}_k \right>dt\Big) \Big],  
\end{multline}  
where 
\begin{gather}
\mathbb{G}_k = [\Id_{n}, -H_k, -\hH_k^{\pi}]^\intercal \hQ_k [\Id_{n}, -H_k, -\hH_k^{\pi}],
\quad
\mathbb{Q}_k = [\Id_{n}, -H_k, -\hH_k^{\pi}]^\intercal Q_k [\Id_{n},  -H_k, -\hH_k^{\pi}],
\notag\\
\mathbb{S}_k = [\Id_{n}, -H_k, -\hH_k^{\pi}]^\intercal  S_k, \quad \bar{\eta}_k = [\Id_{n}, -H_k, \hH_k^{\pi}]^\intercal  Q_k \eta_k, \quad \bar{n}_k = S_k^\intercal \eta_k,\quad
\hH_k^{\pi} = \left[\pi_1 \hH_k, \dots, \pi_K \hH_k \right]. \label{paraminor}
\end{gather}
 
Finally, for the limiting system, we define (i) the major agent's information set $\mc{F}^0:=(\mc{F}^0_t)_{t\in\mc{T}}$ as the filtration generated by $(w^0_{t})_{t\in\mc{T}}$, and (ii) a generic minor agent $i$'s information set $\mc{F}^i\coloneqq (\mc{F}_{t}^i)_{t\in\mc{T}}$ as the filtration generated by $(w^i_{t}, w^0_t)_{t\in\mc{T}}$.
\subsection{Nash Equilibria}
The limiting system described in Section \ref{sec_inf_pop} is a stochastic differential game involving the major agent, the mean field, and the representative minor agent. Our goal is to find the Markovian closed-loop Nash equilibria for this game. We define the admissible set of Markovian closed-loop strategies according to the following assumption.
\begin{assumption}
(Admissible Strategies)\label{ass:MajorControl}
 (i) For the major agent, the set of admissible control strategies  $\mc{U}^{0}$ is defined to be the collection of Markovian linear closed-loop control laws $u^0 \coloneqq (u^0_t)_{t\in\mc{T}}$ such that $\mb{E}[\int_0^T u_t^{0\intercal}u_t^0\, dt] < \infty$. More specifically, $u^0_t = \ell^0_0(t)+\ell^1_0(t)x^0_t+\ell^2_0(t)\bar{x}_t$ for some deterministic functions $\ell^0_0(t), \ell^1_0(t),$ and $\ell^2_0(t)$. (ii) For each minor agent $i,\, i \in \mc{N}$, the set of admissible strategies $\mc{U}^{i}$ is defined to be the collection of Markovian linear closed-loop control laws $u^i\coloneqq (u^i_t)_{t\in\mc{T}}$ such that  $\mb{E}[\int_0^T u_t^{i\intercal}  u_t^i\, dt] < \infty$. More specifically, $u^i_t = \ell^0(t)+\ell^1(t)x^i_t+\ell^2(t)x^0_t+\ell^{3}(t)\bar{x}_t$ for some deterministic functions $\ell^0(t),\ell^1(t), \ell^2(t)$ and $\ell^{3}(t)$.
\end{assumption}
From \eqref{majorExtDynPert0}--\eqref{costinmaj} and \eqref{minorExtDynPert}--\eqref{costmininf}, the major agent's problem involves $\bar{u}$, whereas the representative minor agent's problem involves $u^0$ and $\bar{u}$. Therefore, solving these individual limiting problems requires a fixed-point condition in terms of $\bar{u}$. Note that, under our assumption about the form of an admissible strategy for a representative minor agent, 
if the state and control mean fields exist at the equilibrium, they satisfy the following relation  
\begin{equation}
\bar{u}_t= \lim_{N\rightarrow \infty}\frac{1}{N}\sum_{i \in \mc{N}} u^i_t =\ell^{2}(t)+\ell^{2}(t)x^0_t+(\ell^{1}(t)+\ell^{3}(t))\bar{x}_t. \label{feedpre}
\end{equation}
Thus, we employ the fixed-point approach outlined as follows:
  \begin{itemize}
  \item[(i)] Assume the form given in \eqref{feedpre} for the mean field $\bar{u}$, and solve the resulting differential game given by \eqref{majorExtDynPert0}--\eqref{costinmaj} and \eqref{minorExtDynPert}--\eqref{costmininf} to obtain the best-response strategies  $u^{0,\ast}$ and $u^{i,\ast}$, respectively, for the major agent and a representative minor agent $i$.
    \item[(ii)] Impose the consistency condition 
    $u^{(N)}_{t} = \tfrac{1}{N} \sum_{i\in\mc{N}} u^{i,\ast}_{t}\rightarrow \bar{u}_t$ as $N \rightarrow \infty$ to determine the coefficients in \eqref{feedpre} and fully characterize the mean field equation \eqref{mfeq} that $\bar{x}$ satisfies.
    
To derive the best-response strategies in (i) we use the variational analysis presented in Section \ref{sec_control}. The following theorem summarizes our results.
 \end{itemize}

\begin{theorem} \label{Nash Equilibrium}
[Nash Equilibrium] Suppose Assumptions \ref{assum:majorCost}--\ref{ass:MajorControl} hold. The set of control laws $\{u^{0,*}, u^{i,*}, i \in \mc{N}\}$, where $u^{0,*}$ and $u^{i,*}$ are respectively given by 
\begin{gather} 
     u^{0,*}_{t} = - R_0^{-1} \big [ \mathbb{S}_0^\intercal X^0_{t} -\bar{n}_0+  \mathbb{B}_0^\intercal \big(\Pi_0(t) X^{0}_{t} + s_0(t) \big) \big] \label{majlaw} \\
  u^{i,*}_t = - R_k^{-1} \left[ \mathbb{S}_k^\intercal \, X^{i}_t -\bar{n}_k+  \mathbb{B}_k^\intercal \left(\Pi_k(t) \,X^{i}_t + s_k(t) \right) \right] \label{minorlaw}  
  \end{gather}
  forms a unique Markovian closed-loop Nash equilibrium for the limiting system \eqref{majorExtDynPert0}-\eqref{majorpa} and \eqref{minorExtDynPert}-\eqref{paraminor} subject to the following consistency equations
\begin{align}
\begin{cases} 
&-\dot{\Pi }_{0}=\Pi_{0}\mathbb{A}_{0}+\mathbb{A}_{0}^{\intercal}\Pi_{0}-(\Pi _{0}\mathbb{B}_{0}+\mathbb{S}_{0})R_{0}^{-1}(\mathbb{B}_{0}^{\intercal}\Pi _{0}+\mathbb{S}_{0}^{\intercal})+\mathbb{Q}_{0}+\delta_{0}\Pi _{0}\Sigma _{0}\Sigma _{0}^{\intercal}\Pi _{0}, \quad
\Pi _{0}(T)=\mathbb{G}_{0} \\
&-\dot{\Pi }_{k}=\Pi_{k}\mathbb{A}_{k}+\mathbb{A}_{k}^{\intercal}\Pi_{k}-(\Pi _{k}\mathbb{B}_{k}+\mathbb{S}_{k})R_{k}^{-1}(\mathbb{B}_{k}^{\intercal}\Pi _{k}+\mathbb{S}_{k}^{\intercal})+\mathbb{Q}_{k}+\delta_{k}\Pi _{k}\Sigma _{k}\Sigma _{k}^{\intercal}\Pi _{k},\quad
\Pi _{k}(T)=\mathbb{G}_{k} \\
&\bar{A}_k  = \left[A_k - B_k R_k^{-1} (\mathbb{S}_{k,11}^\intercal + B_k^\intercal \Pi_{k,11})\right] \mathbf{e}_k
+ F^{\pi}_k - B_k R_k^{-1} ( \mathbb{S}_{k,31}^\intercal+ B_k^\intercal \Pi_{k,13}),\\
&\bar{G}_k  = G_k -B_k R_k^{-1}(\mathbb{S}_{k,21}^\intercal + B_k^\intercal \Pi_{k,12}), \label{Pinash}
\end{cases}
\end{align}
\begin{align} \label{snash}
\begin{cases}
&-\dot{s}_{0}=[(\mathbb{A}_{0}-\mathbb{B}_{0}R_{0}^{-1}\mathbb{S}_{0}^{\intercal})^{\intercal}-\Pi _{0}\mathbb{B}_{0}R_{0}^{-1}\mathbb{B}_{0}^{\intercal}]s_{0}+\Pi _{0}(\mathbb{M}_{0}+\mb{B}_0R_{0}^{-1}\bar{n}_0)+\mathbb{S}_{0}R_{0}^{-1}\bar{n}_0 \\&\qquad\hspace{7cm}-\bar{\eta }_{0}+\delta_{0}\Pi _{0}\Sigma _{0}\Sigma _{0}^{\intercal}s_{0}, \quad\quad s_{0}(T)=0\\
&-\dot{s}_{k}=[(\mathbb{A}_{k}-\mathbb{B}_{k}R_{k}^{-1}\mathbb{S}_{k}^{\intercal})^{\intercal}-\Pi _{k}\mathbb{B}_{k}R_{k}^{-1}\mathbb{B}_{k}^{\intercal}]s_{k}+\Pi _{k}(\mathbb{M}_{k}+\mb{B}_kR_{k}^{-1}\bar{n}_k)+\mathbb{S}_{k}R_{k}^{-1}\bar{n}_k \\&\hspace{6.5cm}\qquad \quad-\bar{\eta }_{k}+\delta_{k}\Pi _{k}\Sigma _{k}\Sigma _{k}^{\intercal}s_{k}, \hspace{1cm}
s_{k}(T)=0 \\ 
& \bar{m}_k = b_k+B_kR_k^{-1}\bar{n}_k-B_k R_k^{-1} {B}_k^\intercal s_{k,11}.
\end{cases}
\end{align}
where, for $\Pi_{k}$ and $\mathbb{S}_k$, we use the representation 
 \begin{gather}
\Pi_k =
\begin{bmatrix}
\Pi_{k,11} & \Pi_{k,12} & \Pi_{k,13} \\
\Pi_{k,21} & \Pi_{k,22} & \Pi_{k,23}\\
\Pi_{k,31} & \Pi_{k,32} & \Pi_{k,33}
\end{bmatrix},\quad 
\mathbb{S}_k = \begin{bmatrix} \mathbb{S}_{k,11} \\ \mathbb{S}_{k,21} \\ \mathbb{S}_{k,31}\end{bmatrix}, \quad s_k = \begin{bmatrix} s_{k,11}\\ s_{k,21} \\ s_{k,31}\end{bmatrix} \label{pk}
\end{gather}
with $\Pi_{k,11}, \Pi_{k,22} \in \mbR^{n \times n}$, $\Pi_{k,33} \in \mbR^{nK \times nK}$, $\mathbb{S}_{k,11}, \mathbb{S}_{k,21} \in \mbR^{n \times m}$, $\mathbb{S}_{k,31} \in \mbR^{nK \times m}$, $s_{k,11}, s_{k,21} \in \mbR^{n}$, $s_{k,31} \in \mbR^{nK}$, and 
\begin{equation}
\bar{A} = \begin{bmatrix} \bar{A}_1\\ \vdots\\\bar{A}_K\end{bmatrix}, \quad 
\bar{G} = \begin{bmatrix} \bar{G}_1\\ \vdots\\ \bar{G}_K\end{bmatrix}, \quad 
\bar{m} = \begin{bmatrix} \bar{m}_1\\ \vdots \\ \bar{m}_K \end{bmatrix},
\end{equation}
\begin{equation}
\mb{A}_0 = \begin{bmatrix}
A_0 & F_0^{\pi}\\
\bar{G} & \bar{A}
\end{bmatrix}, 
 \,\, \mb{A}_k = \begin{bmatrix} A_k & [G_k \, \, \, F_k^{\pi}]\\ 0 &\mb{A}_0- \mathbb{B}_0 R^{-1}_0 (\mathbb{S}_0^\intercal+\mb{B}_0^\intercal \Pi_0) \end{bmatrix}, \,\,
\mb{M}_0 = \begin{bmatrix} b_0 \\ \bar{m} \end{bmatrix},\,\,\mb{M}_k = \begin{bmatrix} b_k \\ \mb{M}_0-\mb{B}_0 R_0^{-1} \mb{B}_0^\intercal s_0 \end{bmatrix}.\label{coeff}
\end{equation}
In addition, the mean field $\bar{x}_t$ satisfies
\begin{equation}\label{MF_uperturbed}
d\bar{x}_{t} = \left(\bar{A}\,\bar{x}_{t} + \bar{G}\, x^{0}_{t} + \bar{m}\right)dt.
\end{equation}

\end{theorem}
\begin{proof}
Under Assumption \ref{ass:MajorControl}, the mean field of control actions $\bar{u}_t$ may be expressed as 
\begin{equation} \label{feedbackubar_major}
\bar{u}_{t}=  \Xi(t)X^{0}_{t}+\varsigma(t).
 \end{equation}
where the matrix $\Xi(t)$ and the vector $\varsigma(t)$ are deterministic functions of appropriate dimensions.We begin by examining the major agent's system \eqref{majorExtDynPert0}--\eqref{majorpa}. Using the representation \eqref{feedbackubar_major}, the major agent's extended dynamics may be rewritten as
 \begin{equation} \label{dymajorex}
dX^{0}_{t} = \left((\widetilde{A}_0+\widetilde{B}_{0} \Xi) \,X^{0}_{t} + \mb{B}_0\, u^0_t +\widetilde{M}_{0}+\widetilde{B}_{0}\varsigma(t)\right )dt + \Sigma_{0} dW^{0}_{t}.
 \end{equation}
Subsequently, the optimal control problem faced by the major agent reduces to a single-agent optimization problem. We use the methodology presented in  Section \ref{sec_control} to solve this resulting optimal control problem for the major agent's extended problem. According to Theorem \ref{thm:Cntrl_initial}, the major agent's best-response strategy is given by 
\begin{multline}\label{optmajor1}
u^{0,*}_t =-R_{0}^{-1}\Bigg[\mathbb{S} _{0}^{\intercal}X^{0}_{t}-\bar{n}_0+\mb{B}_0^{\intercal}\Big((\Upsilon ^{-1}_{\widetilde{A}_0+\widetilde{B}_{0}\Xi}(t))^{\intercal}\frac{M_{2,t}^0(u^{0,*})}{M_{1,t}^0(u^{0,*})}\\-(\Upsilon ^{-1}_{\widetilde{A}_0+\widetilde{B}_{0}\Xi}(t))^{\intercal}\int_{0}^{t}(\Upsilon _{\widetilde{A}_0+\widetilde{B}_{0}\Xi}(s))^{\intercal}(\mathbb{Q} _{0}X^{0}_{s}+\mathbb{S} _{0}u^{0,*}_s-\bar{\eta }_{0})ds \Big) \Bigg]
\end{multline} 
where $\Upsilon _{\widetilde{A}_0+\widetilde{B}_{0}\Xi}(t)$ and $ \Upsilon ^{-1}_{\widetilde{A}_0+\widetilde{B}_{0}\Xi}(t)$ are defined as in \eqref{ODE_trans}, and the martingale terms $M_{1,t}^0(u^{0,*})$ and $M_{2,t}^0(u^{0,*})$ are defined as in \eqref{M1t} and \eqref{M2t}, respectively. From Lemma \ref{thm:changeM}, Theorem \ref{thm:cntrl_NewM}, and Corollary \ref{coro4}, we can show that \eqref{optmajor1} admits a unique feedback representation
 \begin{equation}\label{optconmajor}
 u^{0,*}_{t} = - R_0^{-1} \big [ \mathbb{S}_0^\intercal X^0_{t} -\bar{n}_0+  \mathbb{B}_0^\intercal \big(\Pi_0(t) X^{0}_{t} + s_0(t) \big) \big]
 \end{equation}
 where $\Pi_0(t)$ and $s_0(t)$ satisfy 
 \begin{multline} \label{eqmajor}
 \Big(\dot{\Pi}_{0}(t)+\mathbb{Q}_{0}+(\widetilde{A}_0+\widetilde{B}_{0}\Xi(t))^{\intercal}\Pi_{0} (t)+\Pi_{0} (t)(\widetilde{A}_0+\widetilde{B}_{0}\Xi(t))-(\Pi_{0}(t)\mb{B}_0+\mathbb{S} _{0})R_{0}^{-1}(\mb{B}_0^{\intercal}\Pi_{0}(t)+\mathbb{S} _{0}^{\intercal})\\+\delta_{0}\Pi_{0}(t) \Sigma_{0}\Sigma_{0} ^{\intercal}\Pi_{0}(t)\Big)X^{0}_{t} +\Big(\dot{s}_{0}(t)+( (\widetilde{A}_0+\widetilde{B}_{0}\Xi(t))^{\intercal}-\Pi_{0}(t)\mb{B}_0R_{0}^{-1}\mb{B}_0^{\intercal}-\mathbb{S} _{0}R_{0}^{-1}\mb{B}_0^{\intercal}\\+\delta_{0}\Pi_{0}(t) \Sigma_{0}\Sigma_{0} ^{\intercal}  )s_{0}(t)+\Pi_{0} (t)\mb{B}_0R_{0}^{-1}\bar{n}_0 +\Pi_{0} (t)(\widetilde{M}_{0}(t)+\widetilde{B}_{0}\varsigma(t))+\mathbb{S}_{0}R_{0}^{-1}\bar{n}_0 -\bar{\eta }_{0}\Big)=0.    
\end{multline}
This linear-state feedback form is obtained through a change of measure to $\widehat{\mb{P}}^0$, defined by $\frac{d\widehat{\mathbb{P}}^0}{d\mathbb{P}}=\exp(-\frac{\delta_0 ^{2}}{2}\int_{0}^{T}\left\|\gamma_t^0 \right\|^{2}dt+\delta_0 \int_{0}^{T} (\gamma^0_t)^\intercal dW_{t}^0)$, where $\gamma^0_t=\delta_{0}\Sigma _{0}^{\intercal}(\Pi_{0}(t)X^{0}_{t}+s_{0}(t))$. Under the equivalent measure $\widehat{\mb{P}}^0$, the process
\begin{equation} \label{martmajor}
 \frac{M_{2,t}^0(u^{0,*})}{M_{1,t}^0(u^{0,*})}=\widehat{M}^{0}_{t}(u^{0,*})
 \end{equation}
is a martingale, and we have   
\begin{equation}
 (\Upsilon ^{-1}_{\widetilde{A}_0+\widetilde{B}_{0}\Xi}(t))^{\intercal}\widehat{M}^{0}_{t}(u^{0,*})-(\Upsilon ^{-1}_{\widetilde{A}_0+\widetilde{B}_{0}\Xi}(t))^{\intercal}\int_{0}^{t}(\Upsilon _{\widetilde{A}_0+\widetilde{B}_{0}\Xi}(s))^{\intercal}(\mathbb{Q} _{0}X^{0}_{s}+\mathbb{S} _{0}u^{0,*}_s-\bar{\eta }_{0})ds= \Pi_0(t) X^{0}_{t} + s_0(t),  
\end{equation}
where $\Upsilon _{\widetilde{A}_0+\widetilde{B}_{0}\Xi}(t)$ and $\Upsilon ^{-1}_{\widetilde{A}_0+\widetilde{B}_{0}\Xi}(t)$ are defined as in \eqref{ODE_trans}. By applying Ito's lemma to both sides of the above equation and equating the resulting SDEs, we obtain \eqref{eqmajor}.
However, we cannot proceed any further at this point since $\Xi$ and $\varsigma$ are not yet characterized. We hence turn to the problem of a representative minor agent. 
Using the mean-field representation \eqref{feedbackubar_major} and the major agent's best-response strategy \eqref{optconmajor}, the extended dynamics of minor agent $i$ are given by 
\begin{equation} \label{dyminet}
dX^{i}_t=( \grave{A}_{k}(t)X^{i}_t+\mathbb{B}_{k}u^{i}_{t}+\grave{M}_{k}(t))dt+\Sigma _{k}dW^{i}_{t}  \end{equation}
where
\begin{gather}
 \grave{A}_k(t) = \left[ \begin{array}{cc} A_k & [G_k \, \, \, F_k^{\pi}]\\ 0 &\widetilde{A}_0+\widetilde{B}_{0} \Xi-\mb{B}_0R_{0}^{-1}(\mathbb{S} _{0}-\mb{B}_0^{\intercal}\Pi_{0}(t)) \end{array} \right], \,\,\,
  \grave{M}_k(t) = \left[ \begin{array}{c} b_k(t) \\ \widetilde{M}_{0}(t)-\mb{B}_0R_{0}^{-1}\mb{B}_0^{\intercal}s_{0}(t)+\widetilde{B}_{0}\varsigma(t) \end{array}\right].
\end{gather}
From Theorem \ref{thm:Cntrl_initial}, the best-response of minor agent $i$ having a cost functional \eqref{costmininf} is given by 
\begin{equation} \label{optminor1}
u^{i,*}_t =-R_{k}^{-1}\Bigg[\mathbb{S} _{k}^{\intercal}X^{k}_{t}-\bar{n}_k+\mb{B}_k^{\intercal}\left ((\Upsilon^{-1}_{ \grave{A}}(t))^{\intercal}\frac{M_{2,t}^i(u^{i,*})}{M_{1,t}^i(u^{i,*})}-(\Upsilon ^{-1}_{ \grave{A}}(t))^{\intercal}\int_{0}^{t}(\Upsilon _{ \grave{A}_k}(s))^{\intercal}(\mathbb{Q} _{k}X^{i}_{s}+\mathbb{S} _{k}u^{i,*}_s-\bar{\eta }_{k})ds  \right )\Bigg],\end{equation}
where $\Upsilon_{ \grave{A}}(t)$ satisfies \eqref{ODE_trans}, and the martingale terms $M_{1,t}^i(u^{i,*})$ and $M_{2,t}^i(u^{i,*})$ are defined as in \eqref{M1t} and \eqref{M2t}, respectively. Similarly, according to Lemma \ref{thm:changeM}, Theorem \ref{thm:cntrl_NewM}, and Corollary \ref{coro4}, \eqref{optminor1} admits the unique feedback form 
\begin{equation} \label{optconminor}
 u^{i,*}_t = - R_k^{-1} \left[ \mathbb{S}_k^\intercal \, X^{i,*}_t -\bar{n}_k+  \mathbb{B}_k^\intercal \left(\Pi_k(t) \,X^{i,*}_t + s_k(t) \right) \right],
\end{equation}
where 
\begin{multline} \label{eqminor}
\Big(\dot{\Pi}_{k}(t)+\mathbb{Q}_{k}+\grave{A}_k(t)^{\intercal}\Pi_{k} (t)+\Pi_{k} (t)\grave{A}_k(t)-(\Pi_{k}(t)\mb{B}_k+\mathbb{S} _{k})R_{k}^{-1}(\mb{B}_k^{\intercal}\Pi_{k}(t)+\mathbb{S} _{k}^{\intercal})\\+\delta_{k}\Pi_{k}(t) \Sigma_{k}\Sigma_{k} ^{\intercal}\Pi_{k}(t)\Big)X^{i}_{t} +\Big(\dot{s}_{k}(t)+( \grave{A}_k(t)^{\intercal}-\Pi_{k}(t)\mb{B}_kR_{k}^{-1}\mb{B}_k^{\intercal}-\mathbb{S} _{k}R_{k}^{-1}\mb{B}_k^{\intercal}+\delta_{k}\Pi_{k}(t) \Sigma_{k}\Sigma_{k} ^{\intercal}  )s_{k}(t)\\+\Pi_{k} (t)\mb{B}_kR_{k}^{-1}\bar{n}_k +\Pi_{k} (t)\grave{M}_{k}(t)+\mathbb{S}_{k}R_{k}^{-1}\bar{n}_k -\bar{\eta }_{k}\Big)=0.   
\end{multline}
The state feedback form \eqref{optconminor} is obtained through a change of measure to $\widehat{\mb{P}}^i$, defined by $\frac{d\widehat{\mathbb{P}}^i}{d\mathbb{P}}=\exp(-\frac{\delta_k ^{2}}{2}\int_{0}^{T}\left\|\gamma_t^i \right\|^{2}dt+\delta_k \int_{0}^{T} (\gamma^i_t)^\intercal dW_{t}^i)$, with $\gamma^i_t = \delta_{k}\Sigma _{k}^{\intercal}(\Pi_{k}(t)X^{i}_{t}+s_{k}(t))$, such that the process
\begin{equation} \label{martmainor}
 \frac{M_{2,t}^i(u^{i,*})}{M_{1,t}^i(u^{i,*})}=M^{i}_{t}(u^{i,*})
 \end{equation}
is a $\widehat{\mathbb{P}}^{i}$-martingale. 
To continue our analysis, we then characterize $\bar{u}$ by applying the consistency condition (ii). To this end, we represent $\Pi_{k}$ and $\mathbb{S}_k$  in \eqref{optconminor} as in \eqref{pk}. From \eqref{optconminor} and \eqref{pk}, the average control action of a minor agent in subpopulation $k$ is given by 
  \begin{equation}\label{minorCntrlAve}
 u^{(N_k)}_t = - R_k^{-1} \Bigg(\begin{bmatrix} \mathbb{S}_{k,11}^\intercal+B_{k}^\intercal\Pi_{k,11} & \mathbb{S}_{k,21}^\intercal+B_{k}^\intercal\Pi_{k,12} & \mathbb{S}_{k,31}^\intercal+B_{k}^\intercal\Pi_{k,13}\end{bmatrix}\begin{bmatrix}
 x^{(N_k)}_{t} \\
 x^{0}_{t}\\
 \bar{x}_{t} \end{bmatrix} -\bar{n}_k+{B}_k^\intercal s_{k,11} \Bigg).
 \end{equation}
In the limit, as $N_k \rightarrow \infty$, $ u^{(N_k)}_t$ converges in quadratic mean to (see e.g. \citet{KizilkaleTAC2016}) 
 \begin{align}\label{minorCntrlMF}
 \bar{u}^{k}_t = - R_k^{-1} \Bigg(
 \begin{bmatrix} \mathbb{S}_{k,11}^\intercal+B_{k}^\intercal\Pi_{k,11} & \mathbb{S}_{k,21}^\intercal+B_{k}^\intercal\Pi_{k,12} & \mathbb{S}_{k,31}^\intercal+B_{k}^\intercal\Pi_{k,13}\end{bmatrix} \begin{bmatrix}
 \bar{x}^{k}_{t} \\
 x^{0}_{t}\\
 \bar{x}_{t}
 \end{bmatrix} -\bar{n}_k+{B}_k^\intercal s_{k,11} \Bigg).
 \end{align}
We observe that the expression in \eqref{minorCntrlMF} has the same structure as \eqref{feedbackubar_major}. Hence, by comparing these two equations, we specify $\Xi$ and $\varsigma$ in terms of $\Pi_k$ and $s_k,\, k \in \mfK$. We then substitute the obtained expressions for $\Xi$ and $\varsigma$ in \eqref{mfeq}, \eqref{eqmajor}, and \eqref{eqminor} to obtain \eqref{MF_uperturbed}, \eqref{Pinash} and \eqref{snash}. 
\end{proof}

\begin{remark}(Comparison of equilibria in the risk-sensitive and risk-neutral cases)
In risk-neutral MFGs with a major agent, neither the volatility of the major agent nor that of the minor agents explicitly affects the Nash equilibrium. This is not the case for the corresponding risk-sensitive MFGs, where we observe the following:
\begin{itemize}
\item The mean field is influenced by the volatility of the major agent and the volatility of all $K$ types of minor agents.
\item The equilibrium control action of the major agent explicitly depends on its own volatility.
\item The equilibrium control action of a  representative minor agent explicitly depends on its own volatility as well as on the volatility of the major agent.
\item The equilibrium control actions of the major agent and of a representative minor agent are impacted by the volatility of the $K$ types of minor agents through the mean-field equation.
\end{itemize}
Furthermore, in the risk-neutral case, only the first block rows of $\Pi_k$ and $s_k$ impact the equilibrium control, as shown in \citet{firoozi2020convex}. However, in the risk-sensitive case, all the blocks of $\Pi_k$ and $s_k$ have an impact on the equilibrium control actions.
\end{remark}
\subsubsection{Solution of Consistency Equations}
In this section, we discuss the solvability of the set of mean field consistency equations given by \eqref{Pinash}-\eqref{snash}. We note that \eqref{Pinash} represents a set of coupled Riccati equations, which may be solved independently from \eqref{snash}. It is challenging to analytically show the existence and uniqueness of a solution to \eqref{Pinash}. However, given a solution to \eqref{Pinash}, \eqref{snash} may be viewed as a system of coupled first order linear ODEs, for which a unique analytical solution is guaranteed.

Here, we use a numerical scheme to solve the set of consistency equations, \eqref{Pinash}-\eqref{snash}, for a specific system instance. 
To this end, we adapt the iterative method used in \cite{huang2010large}. In particular, in our case $\bar{A}$ and $\bar{G}$ are time-dependent functions defined on $\mathcal{T}$. Our algorithm initializes with arbitrary trajectories for $\bar{A}$ and $\bar{G}$ and iterates until the corresponding trajectories of two consecutive iterations converge, where  the error in iteration $j$ relative to iteration $j-1$ is defined by 
\begin{equation}
\text{error}^{(j)} = \left|\bar{A}^{(j)} - \bar{A}^{(j-1)}\right|_{\infty} + \left|\bar{G}^{(j)} - \bar{G}^{(j-1)}\right|_{\infty}, \quad j = 1, 2, \ldots
\end{equation}
with $\left| \cdot \right|_{\infty }$ denoting the supremum norm and the superscript $(\cdot)$ indicating the iteration number.

We illustrate a particular case of the system described by equations \eqref{dynamicmaj}-\eqref{costmin} where the dynamics and cost functionals are given by 
\begin{align*} 
dx^{0}_{t}&=(-2.5x^{0}_{t}+2.5x^{(N)}_{t}+u^{0}_t)dt+0.5dw^{0}_{t},\notag \\
dx^{i}_{t}&=(-5x^{i}_{t}+2.5x^{(N)}_{t}+2.5x^{0}_{t}+u^{i}_{t})dt+0.5 dw^{i}_{t}, \notag 
\end{align*}
\vspace{-0.6cm}
\begin{align*}
&J^{(N)}_0(u^{0},u^{-0})=\mathbb{E}\left[ \textrm{exp}\left(\int_{0}^{T}\left(10\left( x^{0}_{t}-x^{(N)}_{t} \right )^{2}+(u^{0}_t)^{2}\right)dt\right)\right],\notag \\
&J^{(N)}_i(u^{i},u^{-i})=\mathbb{E}\left [ \textrm{exp}\left(\int_{0}^{T}\left(7\left ( x^{i}_{t}-0.5x^{(N)}_{t}-0.5x^{0}_{t} \right )^{2}+\left ( u^{i}_{t} \right )^{2}\right)dt\right)\right].\notag
\end{align*}
Fig. \ref{barA} and Fig. \ref{barG} illustrate, respectively, the trajectories of $\bar{A}^{(j)}$ and $\bar{G}^{(j)}$ over iterations $j = 1, 2, \ldots, 7$. The trajectories show convergence after just a few iterations. Specifically, the iterative error is reduced to $0.7 \times 10^{-14}$ after ten iterations, starting from the initial trajectories $\bar{A}^{(0)}(t)=\bar{G}^{(0)}(t)=0,\, \forall t \in \mc{T}$. Moreover, for the completeness of numerical solutions, Fig. \ref{barm} depicts the associated trajectories for $\bar{m}^{(j)}$, obtained using a similar iterative method.
\begin{figure}[h] 
\begin{minipage}[t]{.32 \textwidth} 
    \includegraphics[width=\textwidth]{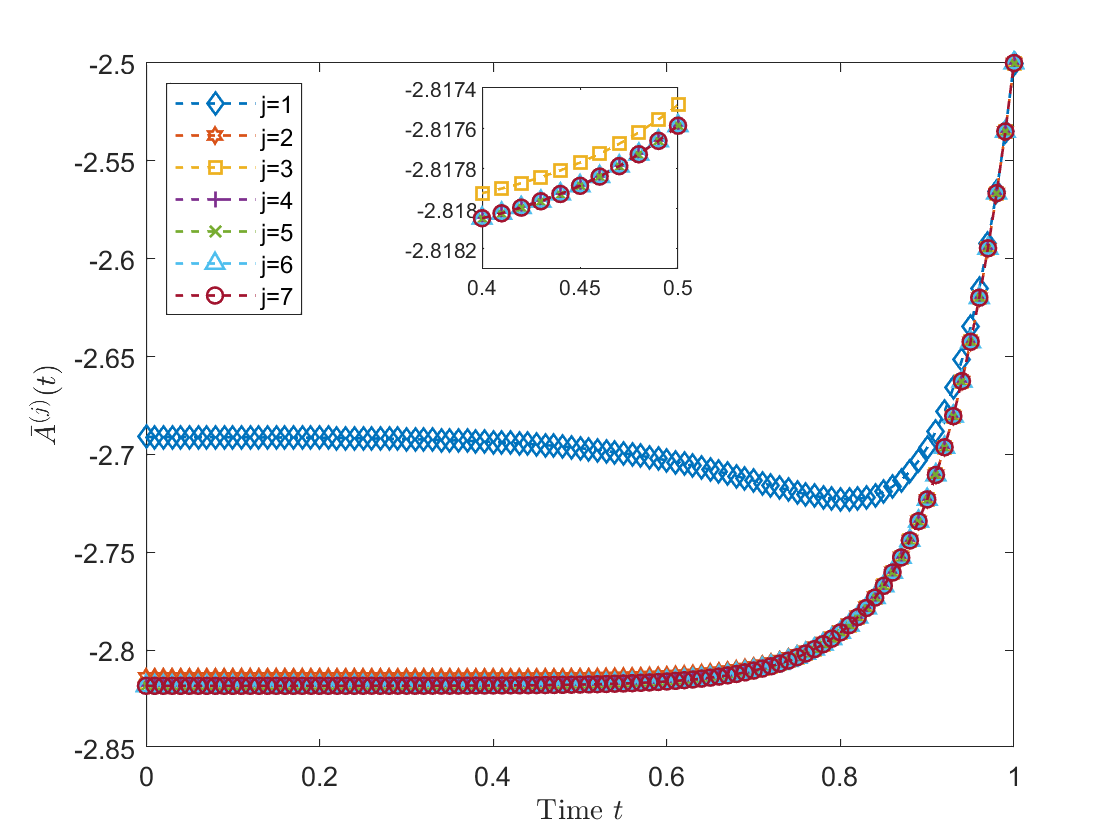}
    \caption{{Trajectories of $\bar{A}^{(j)}$ for $j = 1,\ldots, 7$.}}
    \label{barA}
    \end{minipage}
\hfill 
\begin{minipage}[t]
{.32\textwidth} 

\includegraphics[width=\textwidth]{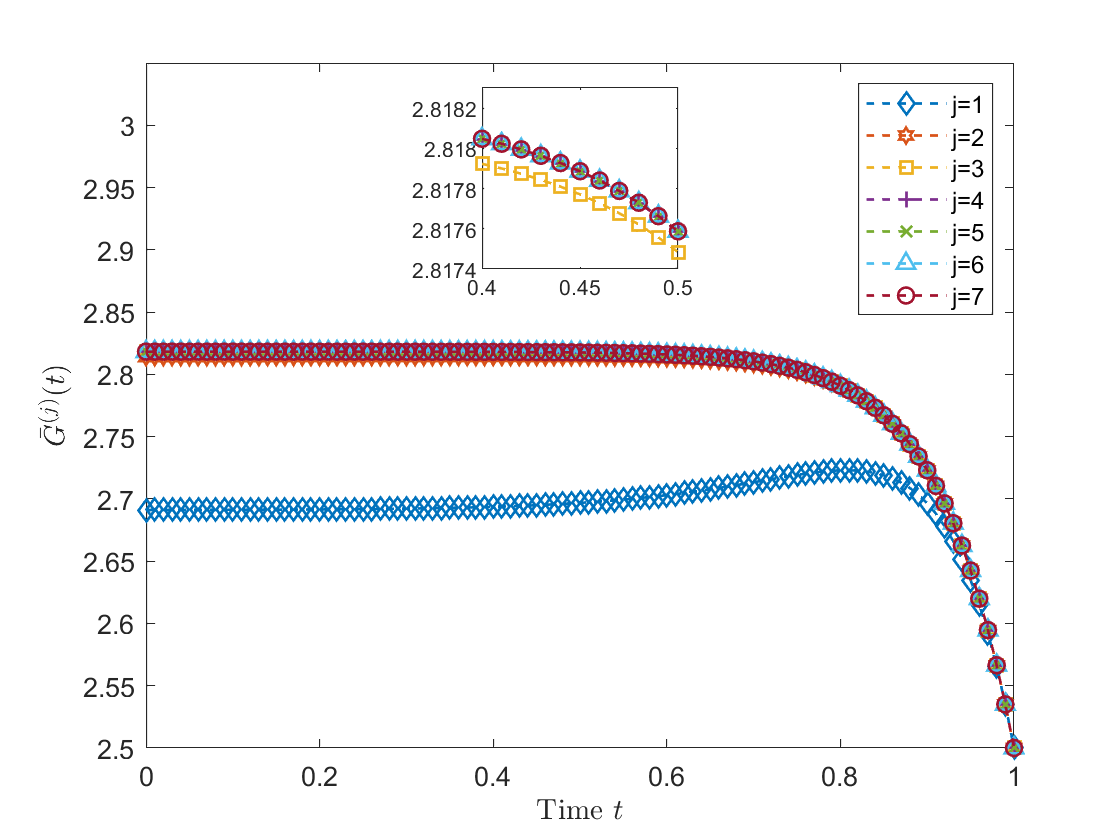}
    \caption{Trajectories of $\bar{G}^{(j)}$ for $j = 1, \ldots, 7$.}
    \label{barG}
     
\end{minipage}
\hfill 
\begin{minipage}[t]{.32\textwidth} 
 \includegraphics[width=\textwidth]{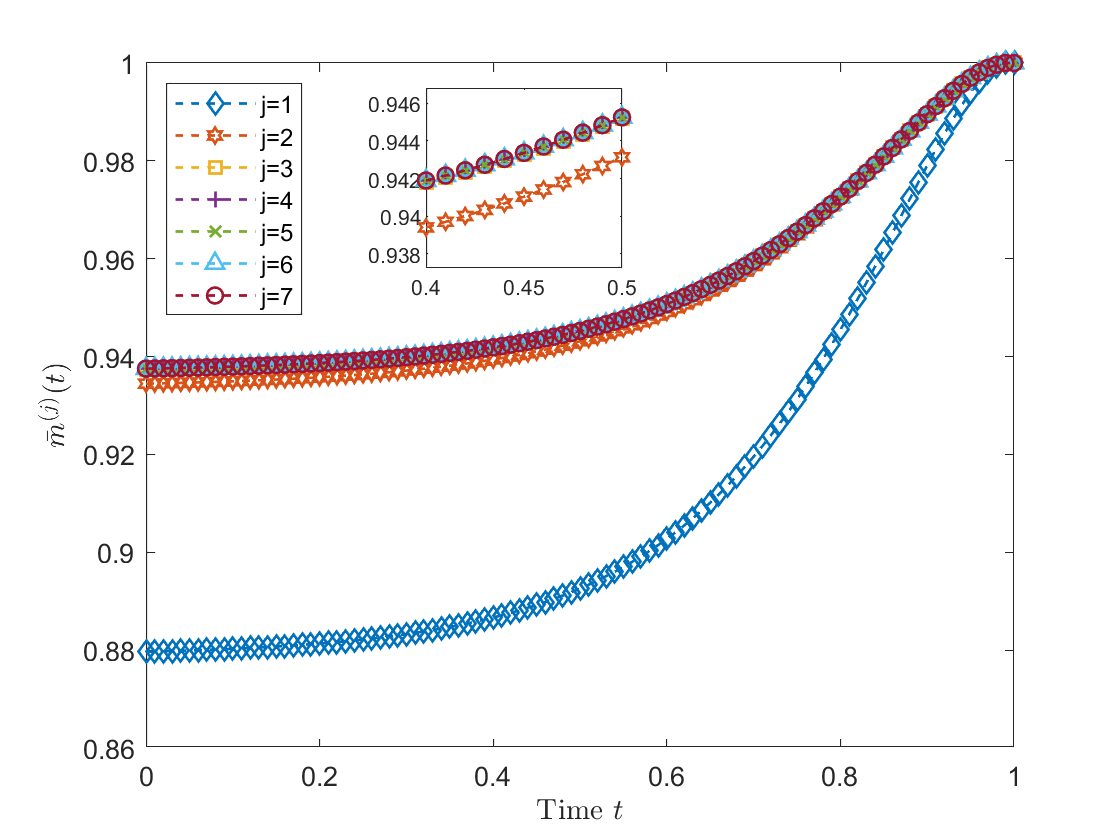}
    \caption{Trajectories of $\bar{m}^{(j)}$ for $j = 1, \ldots, 7$.}
    \label{barm}
\end{minipage}
\end{figure}

\subsection{$\varepsilon$-Nash Property}
\label{sec_epsilon}
In this section, we show that the control laws defined in the previous section yield an $\varepsilon$-Nash equilibrium for the $N$-player game described in Section \ref{sec:finite_pop}  under certain conditions. Due to the fact that linear-quadratic (LQ) exponential cost functional do not admit the boundedness or Liptschitz continuous properties, a proof of the $\varepsilon$-Nash property without imposing further conditions on the system is still an open question. To the best of our knowledge, the only way to establish the $\varepsilon$-Nash property is to find a relationship between linear-quadratic risk-sensitive and risk-neutral cost functionals (see also \citet{moon2019risk}). To be more precise, we first represent the risk-sensitive cost functional as  
\begin{gather}
J(u)=\mathbb{E}\left[ \exp\left(g(x_{T},v_{T}) +\int_{0}^{T}f(x_{t},u_{t},v_{t})dt\right) \right]\\
f(x_{t},u_{t},v_{t})= \frac{\delta}{2}\left(\left< Q(x_{t}-v_{t}),x_{t}-v_{t}\right> +2\left<Su_{t},x_{t}-v_{t} \right>+\left< Ru_{t},u_{t}\right> \right)\\
g(x_{T},v_{T})=\frac{\delta}{2}\left<\widehat{Q}(x_{T}-v_{T}),x_{T}-v_{T}\right>,
\end{gather}
where $v_{t}$ is a square integrable process and we drop the index $i$ for notational brevity. We also assume that $R>0$, $\widehat{Q}\geq 0$, and $Q-SR^{-1}S^{\intercal } \geq 0$.
The desired relationship can be built if the following inequalities hold 
\begin{align} \label{dj}
&\left|J(u_{1})-J(u_{2}) \right|\notag \\ &\leq  \mathbb{E}\Bigg[\Big\vert \exp(g(x^{1}_{T},v^{1}_{T}) +\int_{0}^{T}f(x^{1}_{t},u^{1}_{t},v^{1}_{t})dt)-\exp(g(x^{2}_{T},v^{2}_{T}) +\int_{0}^{T}f(x^{2}_{t},u^{2}_{t},v^{2}_{t})dt) \Big\vert \Bigg]\notag\\ &\leq C\mathbb{E}\Bigg[\Bigg\vert g(x^{1}_{T},v^{1}_{T})- g(x^{2}_{T},v^{2}_{T})+\int_{0}^{T}f(x^{1}_{t},u^{1}_{t},v^{1}_{t})-f(x^{2}_{t},u^{2}_{t},v^{2}_{t})dt\Bigg\vert\Bigg], 
\end{align}
where $C$ is a constant that does not depend on a particular sample path of the stochastic processes involved. However, without any further assumptions, we cannot claim \eqref{dj} since the exponential function and quadratic functions are not globally Lipschitz continuous or bounded. A compromise to address this issue is to confine the control and state vectors within a sufficiently large (unattainable) compact set in Euclidean space without affecting the optimization outcome. This restriction is customary in the context of LQG risk-sensitive problems and appears in one way or another in different methodologies. For more details on this restriction, we refer the reader to \citet{moon2018risk,moon2019risk}, Chapter 6 of \citet{bacsar1998dynamic}, and \citet{lim2005new}, among others.

From \eqref{dj}, the proof of the $\epsilon$-Nash property for LQG risk-sensitive MFGs may be reduced to establishing the same property for LQG risk-neutral MFGs. We refer to \citet{huang2006large} for a related proof for LQG risk-neutral MFGs. In this paper, we leverage the imposed restriction to present an alternative proof. It is easy to see that, under this restriction, we further have
\begin{align} \label{transrs}
 & \mathbb{E}\Bigg[\Bigg\vert g(x^{1}_{T},v^{1}_{T})- g(x^{2}_{T},v^{2}_{T})+\int_{0}^{T}f(x^{1}_{t},u^{1}_{t},v^{1}_{t})-f(x^{2}_{t},u^{2}_{t},v^{2}_{t})dt\Bigg\vert\Bigg]    \notag\\& \leq C\mathbb{E}\Bigg[\left\|x^{1}_{T}-x^{2}_{T} \right\|_{1}+\left\|v^{1}_{T}-v^{2}_{T} \right\|_{1}+\int_{0}^{T}\left(\left\|x^{1}_{t}-x^{2}_{t} \right\|_{1}+\left\|u^{1}_t-u^{2}_t \right\|_{1} +\left\|v^{1}_{t}-v^{2}_t \right\|_{1}\right)dt\Bigg], 
 \end{align}
 where $\left\|. \right\|_{1}$ is the $L^{1}$ vector norm  defined as $\left\| x\right\|_{1}=\sum_{j=1}^{n}\left|x_{j} \right|$, with $x=\left[x_{1}\,\, x_{2}\,\, \dots\, x_{n}\right]^\intercal$. The proof of the $\varepsilon$-Nash property can then be established by using the basic approximation of linear SDEs. 
\begin{theorem}
Suppose that Assumptions \ref{assum:majorCost}--\ref{ass:MajorControl} hold, and that the control and state vectors are restricted within sufficiently large compact sets. Further, suppose that there exists a sequence of real numbers $\{\tau_N, N \in \{1,2,\dots\}\}$ such that $\tau_N \rightarrow 0$ and $\left| \tfrac{N_k}{N} -  \pi_k\right| = o(\tau_N)$, for all $k \in \mc{K}$. The set of control laws $\{u^{0,*}, u^{i,*}, i \in \mc{N}\}$ where $u^{0,*}$ and $u^{i,*}$ are respectively given by  \eqref{majlaw}--\eqref{coeff}, forms an $\varepsilon$-Nash equilibrium for the finite-population system described by \eqref{dynamicmaj}--\eqref{costmin}. That is, for any alternative control action $u^{i} \in \mc{U}^g, i\in \mc{N}_0$, there is a sequence of nonnegative numbers $\{\varepsilon_{N}, N \in \{1,2,\dots\} \}$ converging to zero, such that 
\end{theorem}
\begin{equation}
J_{i}^{(N)}(u^{i,*},u^{-i,*}) \leq   J_{i}^{(N)}(u^{i},u^{-i,*})+\varepsilon_{N} ,\quad  i\in \mc{N}_0
\end{equation}
where $\varepsilon_{N}= o(\tfrac{1}{\sqrt{N}}) + o(\tau_N)$.
\begin{proof}
Without loss of generality, to streamline the notation and facilitate the understanding of the approach, we provide a proof for a simplified scenario, where processes are scalar and where $H_{0}=\delta_{0}=1$ and $H_{k}=\widehat{H}_k=\delta_{k}=1 \forall k\in \mc{K}$. 
The proof may be readily extended to the general system considered in Section \ref{sec:finite_pop}. Additionally, we use the notation $A \lesssim B$ to indicate that $A \leq CB$ for some constant $C$.

We establish the $ \epsilon$-Nash property for the major agent and for a representative minor agent. \\
\textit{(I) $ \epsilon$-Nash property for the major agent}: Under the assumption that all minor agents follow the Nash equilibrium strategies $\{u^{i,*}, i \in \mc{N}\}$ given in Theorem \ref{Nash Equilibrium} and that the major agent adopts an arbitrary strategy $u^0$, we introduce the following finite-population and infinite-population systems for the major agent. These systems share the same initial conditions.
\begin{itemize}
\item \textit{Infinite-Population System:} The major agent's state and the mean field are represented, respectively, by $x^{0}_{t}$ and $(\bar{x}_{t})^\intercal := [(\bar{x}^{1}_{t})^\intercal, \dots, (\bar{x}^{K}_{t})^\intercal]$. The dynamics are given by
 \begin{align} \label{majoyinf}
\begin{cases}
dx^{0}_{t} = &[A_0\, x^{0}_{t} + F_{0}^{\pi}\bar{x}_{t} + B_0u^0_{t}]dt + \sigma_0(t)\, dw^0_t\\
d\bar{x}^{k}_{t}=&[(A_{k}-R^{-1}_{k}(B_{k})^{2}(\Pi_{k,11}(t)+\mathbb{S} _{k,11}))\bar{x}^{k}_{t}-R^{-1}_{k}(B_{k})^{2}(\Pi_{k,13}(t)+\mathbb{S} _{k,13}^{\intercal})\bar{x}_{t}\\&+\bar{G}_{k}x_{t}^{0}+F_{k}^{\pi}\bar{x}_{t}+\bar{m}_{k}(t)]dt, \qquad k \in \mc{K},  
\end{cases}  
\end{align}
where the coefficients $\bar{G}_k$ and $\bar{m}_k$ are given, respectively, by \eqref{Pinash} and \eqref{snash}. This system contains $K+1$ equations. We represent it in expanded form to facilitate the analysis. For this case, the major agent's cost functional is given by
\begin{align} \label{l_0}
&J_{0}^{\infty }(u^{0},u^{-0,*})=\mathbb{E}\left[ \exp\left(\int_{0}^{T}f^{0}(x^{0}_{t},u^{0}_{t},\pi \bar{x}_{t}+\eta_{0})dt+g^{0}(x^{0}_{T},\pi \bar{x}_{T}+\eta_{0})\right) \right]  \notag \\ 
&f^{0}(x_{t},u_{t},v_{t})= \frac{1}{2} Q_{0}(x_{t}-v_{t})^{2} +S_{0}u_{t}(x_{t}-v_{t})+\frac{1}{2}R_{0}u_{t}^{2} \notag \\  & g^{0}(x_{T},v_{T})=\widehat{Q}_{0}(x_{T}-v_{T})^{2}.
\end{align}    
    \item \textit{Finite-Population System:} This system consists of one major agent and $N$ minor agents. We represent the major agent's state by $x^{0,(N)}_{t}$ and the vector of average states by $(x^{[N]})^{\intercal}= [(x^{(N_1)})^{\intercal}, (x^{(N_2)})^{\intercal},\dots, (x^{(N_K)})^{\intercal}]$ (as defined in \eqref{subave}). Furthermore, we introduce the vector process $(\bar{x}^{(N)})^{\intercal}= [\bar{x}^{1,(N)}_{t}, \bar{x}^{2,(N)}_{t},\dots, \bar{x}^{K,(N)}_{t}]$, which is calculated using the mean-field equation \eqref{mfeq}, where the vector of average states $x^{[N]}_t$ is used. These processes satisfy
\begin{align}
\begin{cases} 
dx^{0,(N)}_{t} &= [A_0\, x^{0,(N)}_{t} + F_{0}^{\pi^{(N)}}x_{t}^{[N]} + B_0u^{0}_{t}]dt + \sigma_0 (t)dw^0_t\\
dx^{(N_{k})}_{t}&=[(A_{k}-R_{k}^{-1}(B_{k})^{2}(\Pi_{k,11}(t)+\mathbb{S} _{k,11}))x^{(N_{k})}_{t}-R_{k}^{-1}(B_{k})^{2}(\Pi_{k,13}(t)+\mathbb{S} _{k,13}^{\intercal})\bar{x}^{(N)}_{t}\\&\hspace{1cm}+\bar{G}_{k}x^{0,(N)}_{t}+F_{k}^{\pi^{(N)}}x_{t}^{[N]}+\bar{m}_{k}(t)]dt +\frac{1}{N_{k}}\sum_{i\in \mc{I}_{k}}\sigma _{k}(t)dw^{i}_{t}, \qquad k\in \mc{K}\\ 
d\bar{x}^{k,(N)}_{t}&=[(A_{k}-R^{-1}_{k}(B_{k})^{2}(\Pi_{k,11}(t)+\mathbb{S} _{k,11}))\bar{x}^{k,(N)}_{t}-R^{-1}_{k}(B_{k})^{2}(\Pi_{k,13}(t)+\mathbb{S} _{k,13}^{\intercal})\bar{x}^{(N)}_{t}\\ &\hspace{5cm}+\bar{G}_{k}x_{t}^{0,(N)}+F_{k}^{\pi}\bar{x}^{(N)}_{t}+\bar{m}_{k}(t)]dt, \qquad k \in \mc{K}, \label{majorN}
\end{cases}
\end{align}
where the dynamics of $x^{(N_{k})}$ are obtained using \eqref{dynamicmin}, and the Nash equilibrium strategies using $x^{i,(N)}_{t}$, $x^{0,(N)}_{t}$, and $\bar{x}^{(N)}_{t}$ instead of, respectively, $x^{i}_{t}$, $x^{0}_{t}$, and $\bar{x}_{t}$.

The major agent's cost functional is then given by 
\begin{equation} \label{diffJ0}
J_{0}^{(N)}(u^{0},u^{-0,*})=\mathbb{E}\left[ \exp\left( \int_{0}^{T}f^{0}(x^{0,(N)}_{t},u^{0}_{t},\pi^{(N)}x_{t}^{[N]}+\eta_{0})dt+g^{0}(x^{0,(N)}_{T},\pi^{(N)}x_{T}^{[N]}+\eta_{0})\right) \right].  
\end{equation}
\end{itemize}
From \eqref{transrs}, we have 
\begin{align}
&\left|J_{0}^{(N)}(u^{0},u^{-0,*})-J_{0}^{\infty }(u^{0},u^{-0,*}) \right|\notag \\ & \lesssim \int_{0}^{T}\mathbb{E}\left [\left|x^{0}_{t}-x^{0,(N)}_{t} \right|+\left|\pi \bar{x}_{t}-\pi^{(N)}x_{t}^{[N]}  \right|\right ]dt+\mathbb{E}\left| \pi \bar{x}_{T}-\pi^{(N)}x_{T}^{[N]} \right|+\mathbb{E}\left|x^{0}_{T}-x^{0,(N)}_{T} \right|.\label{inter1}
\end{align}
Moreover, we can write 
\begin{align} \label{pixt}
\left| \pi \bar{x}_{t}-\pi^{(N)}x_{t}^{[N]} \right|\leq \left|\pi \bar{x}_{t}-\pi x_{t}^{[N]} \right|+ \left|\pi x_{t}^{[N]}-\pi^{(N)}x_{t}^{[N]}  \right|\leq \left\|\bar{x}_{t}- x_{t}^{[N]}\right\|_{1}+C \tau _{N} 
\end{align}
for all $t \in \left [ 0,T \right ]$, where $\tau_{N}:=\sup_{1\leq i\leq K} \left|\pi^{(N)}_{k}-\pi_{k} \right|$ is a sequence converging to zero. Using \eqref{pixt}, \eqref{inter1} may be written as
\begin{align} \label{J0N}
&\left|J_{0}^{(N)}(u^{0},u^{-0,*})-J_{0}^{\infty }(u^{0},u^{-0,*}) \right|\notag \\ & \lesssim \int_{0}^{T}\mathbb{E}\left [ \left|x^{0}_{t}-x^{0,(N)}_{t} \right|+\left\|\bar{x}_{t}- x_{t}^{[N]}\right\|_{1}+\tau _{N}\right ]dt+\mathbb{E}\left\|  \bar{x}_{T}-x_{T}^{[N]} \right\|_{1}+\mathbb{E}\left|x^{0}_{T}-x^{0,(N)}_{T} \right|+\tau _{N}\notag \\ &\lesssim \int_{0}^{T}\xi^{0}_{N}(t)dt+\xi^{0}_{N}(T),
\end{align}
where
\begin{equation}  
 \label{gnt}
\xi^{0}_{N}(t):= \mathbb{E}\left|x^{0}_{t}-x^{0,(N)}_{t} \right|+ \mathbb{E}\left\|\bar{x}_{t}-\bar{x}_{t}^{(N)} \right\|_{1}+\mathbb{E}\left\|\bar{x}_{t}^{(N)}- x_{t}^{[N]}\right\|_{1}+\tau_{N}
\end{equation}
and $\xi^{0}_{N}(0)=\tau_{N}$. We aim to find an upper bound for $\xi^{0}_{N}(t)$. From the second and third equations in \eqref{majorN}, we have
\begin{align}
 \mathbb{E}\left\|\bar{x}_{t}^{(N)}- x_{t}^{[N]}\right\|_{1}&=\sum_{k=1}^{K}\mathbb{E}\left|\bar{x}^{k,(N)}_{t}-x^{(N_{k})}_{t} \right| \notag \\ & \leq  \int_{0}^{t}\Bigg[\sum_{k=1}^{K}\mathbb{E} \left|(A_{k}-R_{k}^{-1}(B_{k})^{2}(\Pi_{k,11}(s)+\mathbb{S} _{k,11}))(\bar{x}^{k,(N)}_{s}-x^{(N_{k})}_{s}) \right| \notag\\&+\sum_{k=1}^{K}\mathbb{E}\left|F_{k}^{\pi}\bar{x}^{(N)}_{s}-F_{k}^{\pi^{(N)}}x_{s}^{[N]} \right| \Bigg]ds+\sum_{k=1}^{K}\frac{1}{N_{k}} \mathbb{E}\left|\sum_{i\in \mc {I}_k }\int_{0}^{t}\sigma _{k}(s)dw^{i}_{s} \right|.\label{step1}   
\end{align}
Since the coefficients are bounded, we have 
\begin{equation}
     \sum_{k=1}^{K}\mathbb{E} \left|(A_{k}-R_{k}^{-1}(B_{k})^{2}(\Pi_{k,11}(t)+\mathbb{S} _{k,11}))(\bar{x}^{k,(N)}_{t}-x^{(N_{k})}_{t}) \right| \lesssim  \mathbb{E}\left\|\bar{x}_{t}^{(N)}- x_{t}^{[N]}\right\|_{1}.\label{step2} 
\end{equation} 
Moreover, following the same approach as in \eqref{pixt}, we obtain
\begin{equation}
\sum_{k=1}^{K}\mathbb{E}\left|F_{k}^{\pi}\bar{x}^{(N)}_{t}-F_{k}^{\pi^{(N)}}x_{t}^{[N]} \right| \lesssim  \mathbb{E}\left\|\bar{x}_{t}^{(N)}- x_{t}^{[N]}\right\|_{1}+\tau_{N}.\label{step3} 
\end{equation}
Therefore, we have 
\begin{align}
\mathbb{E}\left\|\bar{x}_{t}^{(N)}- x_{t}^{[N]}\right\|_{1} & \leq C\int_{0}^{t} \left [\mathbb{E}\left\|\bar{x}_{s}^{(N)}- x_{s}^{[N]}\right\|_{1}+\tau_{N}\right ]ds+\sum_{k=1}^{K}\frac{1}{N_{k}} \mathbb{E}\left|\sum_{i\in \mc {I}_k }\int_{0}^{t}\sigma _{k}(s)dw^{i}_{s} \right| \notag \\ &\leq
  C\int_{0}^{t}\xi^{0}_{N}(s)ds+\sum_{k=1}^{K}\frac{1}{N_{k}} \mathbb{E}\left|\sum_{i\in \mc {I}_k }\int_{0}^{t}\sigma _{k}(s)dw^{i}_{s}.\right|\label{step4}
\end{align}
By following similar steps as in \eqref{step1}--\eqref{step4} for $\mathbb{E}\left|x^{0}_{t}-x^{0,(N)}_{t} \right|$ and $\mathbb{E}\left\|\bar{x}_{t}-\bar{x}_{t}^{(N)} \right\|_{1}$, we obtain the inequality
\begin{equation} \label{gnineq}
\xi^{0}_{N}(t) \lesssim \kappa^{0} _{N}(t) +\int_{0}^{t}\xi^{0}_{N}(s)ds    
\end{equation}
where 
\begin{equation}
\kappa^{0}_{N}(t)=\tau_{N}+\sum_{k=1}^{K}\frac{1}{N_{k}} \mathbb{E}\left|\sum_{i\in \mc{I}_k }\int_{0}^{t}\sigma _{k}(s)dw^{i}_{s} \right|.    
\end{equation}
Note that $\sum_{i\in \mc{I}_k }\int_{0}^{t}\sigma _{k}(s)dw_{s}^i\sim \mathcal{N}(0,N_{k}\int_{0}^{t}\sigma _{k}^{2}(s)ds)$ and hence $\left|\sum_{i\in \mc{I}_k }\int_{0}^{t}\sigma _{k}(s)dw^{i}_{s} \right|$ follows a folded normal distribution. We then have
\begin{equation}
\frac{1}{N_{k}} \mathbb{E}\left|\sum_{i\in \mc{I}_k }\int_{0}^{t}\sigma _{k}(s)dw^{i}_{s} \right|\lesssim\frac{1}{\sqrt{N_{k}}} \lesssim \frac{1}{\sqrt{N}}, 
\end{equation}
where the second inequality holds due to the boundedness of $\frac{1}{\sqrt{\pi^{(N)}_{k}}}$, so that  
\begin{equation}
\kappa^{0} _{N}(t)\lesssim \tau_{N}+\frac{1}{\sqrt{N}},  \qquad \forall t\in \mc{T}.
\end{equation}
We use Grönwall's inequality to write \eqref{gnineq} as  
\begin{equation}
 \xi^{0}_{N}(t) \leq C\kappa ^{0}_{N}(t)+Ce^{Ct}\int_{0}^{t}e^{-Cs}\kappa ^{0} _{N}(s)ds.  
\end{equation}
The above expression indicates that $\{\xi^{0}_{N}(t), N=1, 2, \dots\}$ forms a sequence converging to $0$ as $N \rightarrow \infty$. Therefore, from \eqref{J0N}, we have
\begin{equation} \label{epmaj}
J_{0}^{(N)}(u^{0},u^{-0,*}) \leq J_{0}^{\infty }(u^{0},u^{-0,*})+ \varepsilon_{N},
\end{equation}
where $\varepsilon_{N}= o(\tfrac{1}{\sqrt{N}}) + o(\tau_N)$. Note that \eqref{epmaj} holds for both the optimal strategy $u^{0,*}$ and an arbitrary strategy $u^{0}$ for the major agent. Furthermore, we have
\begin{align}
J_{0}^{(N)}(u^{0,*},u^{-0,*}) &\leq J_{0}^{\infty }(u^{0,*},u^{-0,*})+ \varepsilon_{N} \notag\\&\leq  J_{0}^{\infty }(u^{0},u^{-0,*})+ \varepsilon_{N}\notag\\& \leq J_{0}^{(N)}(u^{0},u^{-0,*})+2 \varepsilon_{N}, 
\end{align}
where the second inequality holds since $u^{0,*}$ represents the Nash equilibrium strategy for the infinite-population system, and the third inequality follows from \eqref{J0N}. This concludes the proof of the $\epsilon$-Nash property for the major agent.

\textit{(II) $ \epsilon$-Nash property for the representative minor agent $i$}: 
Assuming that the major agent and all minor agents, except minor agent $i$ in subpopulation $1$, follow the Nash equilibrium strategies $\{u^{0,*}, u^{1,*},\dots, u^{i-1,*}, u^{i+1,*},\dots, u^{i+1,*}\}$ as outlined in Theorem \ref{Nash Equilibrium}, while minor agent $i$ adopts an arbitrary strategy $u^i$, we introduce the following finite-population and infinite-population systems for minor agent $i$. These systems share the same initial conditions.

\begin{itemize}
\item \textit{Infinite-Population System:} The minor agent's state, the major agent's state and the mean field are represented, respectively, by $x^i_t$, $x^{0}_{t}$ and $(\bar{x}_{t})^\intercal := [(\bar{x}^{1}_{t})^\intercal, \dots, (\bar{x}^{K}_{t})^\intercal]$. The dynamics are given by
\begin{align} \label{minorinf}
\begin{cases}
dx^{i}_{t}&=[A_{1}x^{i}_{t}+F_{1}^{\pi}\bar{x}_{t}+G_{1}x^{0}_{t}+B_{1}u^{i}_{t}]dt+\sigma _{1}(t)dw^{i}_{t}\\
dx^{0}_{t} &= [(A_0-R_{0}^{-1}B_{0}^{2}(\Pi _{0,11}(t)+\mathbb{S} _{0,11}))x^{0}_{t} -R_{0}^{-1}B_{0}^{2}(\Pi _{0,12}(t)+\mathbb{S} _{0,12}^{\intercal})\bar{x}_{t}\\
&\hspace{3.5cm}+F_0^{\pi}\bar{x}_{t}+b_0+B_0R_k^{-1}\bar{n}_0-R_0^{-1} (B_0)^{2} s_0]dt + \sigma_0(t)dw^{0}_{t} \\
d\bar{x}^{k}_{t}&=[(A_{k}-R^{-1}_{k}(B_{k})^{2}(\Pi_{k,11}(t)+\mathbb{S} _{k,11}))\bar{x}^{k}_{t}-R^{-1}_{k}(B_{k})^{2}(\Pi_{k,13}(t)+\mathbb{S} _{k,13}^{\intercal})\bar{x}_{t}\\
&\hspace{8cm}+\bar{G}_{k}x_{t}^{0}+F_{k}^{\pi}\bar{x}_{t}+\bar{m}_{k}(t)]dt,   \\
\end{cases}     
\end{align}
where the coefficients $\bar{G}_k$ and $\bar{m}_k$ are, respectively, given by \eqref{Pinash} and \eqref{snash}. 
 The cost functional of minor agent $i$ belonging to subpopulation $1$ is given by
\begin{align}
&J_{i}^{\infty }(u^{i},u^{-i,*})=\mathbb{E}[ \exp(\int_{0}^{T}f^{i}(x^{i}_{t},u^{i}_{t},\pi \bar{x}_{t}+x_{t}^{0}+\eta_{1})dt+g^{i}(x^{i}_{T},\pi \bar{x}_{T}+x_{T}^{0}+\eta_{1})) ]  \notag \\ &
f^{i}(x^{i}_{t},u^{i}_{t},v_{t})= \frac{1}{2} Q_{1}(x^{i}_{t}-v_{t})^{2}+S_{1}u_{t}^{i}(x^{i}_{t}-v_{t}) +\frac{1}{2}R_{1}(u^{i}_{t})^{2}\notag \\ &
g^{i}(x^{i}_{T},v_{T})=\widehat{Q}_{1}(x^{i}_{T}-v_{T})^{2}.
\end{align}  
\item \textit{Finite-Population System:} This system consists of one major agent and $N$ minor agents. We represent the state of minor agent $i$ belonging to subpopulation $1$ by $x^{i,(N)}_{t}$, the major agent's state by $x^{0,(N)}_{t}$, the average state of subpopulation $1$ by $x^{(N_1)}$, and the average state of subpopulation $k\neq 1$ by $x^{(N_k)}$ (as defined in \eqref{subave}). Furthermore, we introduce the vector process $(\bar{x}^{(N)})^{\intercal}= [\bar{x}^{1,(N)}_{t}, \bar{x}^{2,(N)}_{t},\dots, \bar{x}^{K,(N)}_{t}]$, which is calculated using the mean-field equation \eqref{mfeq}, using the vector of average states $(x^{[N]})^{\intercal}= [(x^{(N_1)})^{\intercal}, (x^{(N_2)})^{\intercal},\dots, (x^{(N_K)})^{\intercal}]$. These processes satisfy
\begin{align} \label{minorN}
\begin{cases}
dx^{i,(N)}_{t}&=[A_{1}x^{(N_{1})}_{t}+F_{1}^{\pi^{(N)}}\bar{x}_{t}^{(N)}+G_{1}x^{0,(N)}_{t}+B_{1}u^{i}_{t}]+\sigma_{1}(t)dw^{i}_{t}\\
dx^{0,(N)}_{t} &= [(A_0-R_{0}^{-1}B_{0}^{2}(\Pi _{0,11}(t)+\mathbb{S} _{0,11}))x^{0,(N)}_{t} -R_{0}^{-1}B_{0}^{2}(\Pi _{0,12}(t)+\mathbb{S} _{0,12}^{\intercal})\bar{x}_{t}^{(N)}\\&\hspace{3cm}+F_0^{\pi^{(N)}}x_{t}^{[N]}+\bar{m}_0(t)]dt + \sigma_0(t)dw^{0}_{t} \\
dx^{(N_{1})}_{t}&=[(A_{1}-R_{1}^{-1}B_{1}^{2}(\Pi_{1,11}(t)+\mathbb{S} _{1,11}))x^{(N_{1})}_{t}-R_{1}^{-1}B_{1}^{2}(\Pi_{1,13}(t)+\mathbb{S} _{1,13}^{\intercal})\bar{x}^{(N)}_{t}\\&\hspace{1cm}+\bar{G}_{1}x^{0,(N)}_{t}+F_{1}^{\pi^{(N)}}x_{t}^{[N]}+\bar{m}_{1}(t)]dt +\frac{1}{N_{1}}\sum_{i\in \mc{I}_{1}}\sigma _{1}(t)dw^{i}_{t}+e_t^{u^{i,*}, u^{i}}dt\\
dx^{(N_{k})}_{t}&=[(A_{k}-R_{k}^{-1}(B_{k})^{2}(\Pi_{k,11}(t)+\mathbb{S} _{k,11}))x^{(N_{k})}_{t}-R_{k}^{-1}(B_{k})^{2}(\Pi_{k,13}(t)+\mathbb{S} _{k,13}^{\intercal})\bar{x}^{(N)}_{t}\\&\hspace{2cm}+\bar{G}_{k}x^{0,(N)}_{t}+F_{k}^{\pi^{(N)}}x_{t}^{[N]}+\bar{m}_{k}(t)]dt +\frac{1}{N_{k}}\sum_{i\in \mc{I}_{k}}\sigma _{k}(t)dw^{i}_{t},\qquad k \neq 1\\
d\bar{x}^{k,(N)}_{t}&=[(A_{k}-R^{-1}_{k}(B_{k})^{2}(\Pi_{k,11}(t)+\mathbb{S} _{k,11}))\bar{x}^{k,(N)}_{t}-R^{-1}_{k}(B_{k})^{2}(\Pi_{k,13}(t)+\mathbb{S} _{k,13}^{\intercal})\bar{x}^{(N)}_{t}\\ &\hspace{5cm}+\bar{G}_{k}x_{t}^{0,(N)}+F_{k}^{\pi}\bar{x}^{(N)}_{t}+\bar{m}_{k}(t)]dt,
\end{cases}  
\end{align}
where the average state $x^{(N_{1})}_{t}$ of subpopulation $1$ involves the term 
\begin{align}
e_t^{u^{i,*}, u^{i}}:=-\frac{B_{1}}{N_{1}}( u^{i,*}_{t}-u^{i}_{t}),
\end{align}
where $u^{i,*}$ represents the Nash equilibrium strategy of agent $i$, as given by \eqref{optconminor} for $k=1$. The inclusion of this term accounts for the arbitrary strategy adopted by minor agent $i$. Notably, this term becomes zero when the agent chooses to follow the Nash equilibrium strategy. The cost functional of minor agent $i$ is given by
\begin{align}
J_{i}^{(N)}(u^{i},u^{-i,*})=\mathbb{E}\Big[ \exp\Big(\int_{0}^{T}f^{i}(x^{i,(N)}_{t},  u^{i}_{t}, &\pi^{(N)}x_{t}^{[N]}+x^{0,(N)}_{t}+\eta_{1})dt\notag \\ &+g^{i}(x^{i,(N)}_{T},\pi^{(N)}x_{T}^{[N]}+x^{0,(N)}_{T}+\eta_{1})\Big)\Big].  
\end{align}  
\end{itemize}
Following a similar approach as for the major agent's problem, we obtain
\begin{align}
\left|J_{i}^{(N)}(u^{i},u^{-i,*})-J_{0}^{\infty }(u^{i},u^{-i,*}) \right| \notag &\lesssim  \int_{0}^{T}\mathbb{E}\left [\left|x^{i}_{t}-x^{i,(N)}_{t} \right|+\left|x^{0}_{t}-x^{0,(N)}_{t} \right|+\left|\pi \bar{x}_{t}-\pi^{(N)}x_{t}^{[N]}  \right|\right ]dt\notag\\ &\hspace{0.5cm}+\mathbb{E}\left|x^{i}_{T}-x^{i,(N)}_{T} \right| +\mathbb{E}\left| \pi \bar{x}_{T}-\pi^{(N)}x_{T}^{[N]} \right|+\mathbb{E}\left|x^{0}_{T}-x^{0,(N)}_{T} \right| \notag\\
&\lesssim  \int_{0}^{T}\xi^{i}_{N}(t)dt +\xi^{i}_{N}(T), 
\end{align}
where
\begin{align}
\xi^{i}_{N}(t):= &\mathbb{E}\left|x^{i}_{t}-x^{i,(N)}_{t} \right|+\mathbb{E}\left|x^{0}_{t}-x^{0,(N)}_{t} \right|+ \mathbb{E}\left\|\bar{x}_{t}-\bar{x}_{t}^{(N)} \right\|_{1}+\mathbb{E}\left\|\bar{x}_{t}^{(N)}- x_{t}^{[N]}\right\|_{1}+\tau_{N}.
\end{align}
Again, as in the case of the major agent, we have 
\begin{equation}
\xi^{i}_{N}(t) \leq C\kappa^{i}_{N}(t)+Ce^{Ct}\int_{0}^{t}e^{-Cs}\kappa^{i} _{N}(s)ds,    
\end{equation}
where
\begin{align}
\kappa^{i}_{N}(t)=&\tau_{N}+\sum_{k=1}^{K}\frac{1}{N_{k}} \mathbb{E}\left|\sum_{i\in \mc{I}_k }\int_{0}^{t}\sigma _{k}(s)dw^{i}_{s} \right| +\mathbb{E}\left|e^{u^{i,*}, u^{i}}_t \right|.
\end{align}
Note that, according to our assumptions,  
\begin{equation}
\mathbb{E}\left|e^{u^{i,*}, u^{i}}_t \right| \lesssim \frac{1}{N}.   \end{equation}
Thus, $\kappa^{i}_{N}(t) \rightarrow 0$ as $N \rightarrow \infty,\, \forall t\in \left [ 0,T \right ]$. Then, following similar steps as in the major agent's problem, we can show that $J_{i}^{(N)}(u^{i,*},u^{-i,*}) \leq   J_{i}^{(N)}(u^{i},u^{-i,*})+\varepsilon_{N}$, where $\varepsilon_N = o(\tfrac{1}{\sqrt{N}}) + o(\tau_N)$.
\end{proof}

\subsection{A toy model}
We consider a toy model in this section, which is supposed to be the simplest one can appear. The dynamic of the system is given by 
\begin{align} \label{simpex}
&dx^{0}_{t}=(x^{0}_{t}+x^{(N)}_{t}+u^{0}_t)dt+\sigma _{0}dw^{0}_{t}, 
 J_{0}^{(N)}(u^{0},u^{-0})=\mathbb{E} \bigg[\exp\frac{1}{2}\int_{0}^{T}(x^{0}_{t}-x^{(N)}_{t})^2+(u^{0}_{t})^2dt\bigg)\bigg] \\ 
&dx^{i}_{t}=(x^{i}_{t}+x^{(N)}_{t}+x^{0}_{t}+u^{i}_{t})dt+\sigma dw^{i}_{t}, J_{i}^{(N)}(u^{i},u^{-i})=\mathbb{E}\bigg[ \mathrm{exp}\frac{1}{2}\int_{0}^{T}(x^{i}_{t}-x^{0}_{t}-x^{(N)}_{t})^2 +(u^{i}_{t})^2 dt\bigg]  \notag
\end{align}
We only have one type of minor agent here and the dimension is just one for all agents. The pre-assumed $\Xi(t),\varsigma(t)$ in \eqref{feedbackubar_major} are just a $\mbR^2$ and real valued functions,respectively. Let us assume $\Xi (t)=\begin{bmatrix}
\Xi_1 (t) & \Xi_2 (t)
\end{bmatrix}$. We now look at the major agent's problem (see \eqref{dymajorex}-\eqref{eqmajor}) according to the dynamic.
 \begin{equation}
dX^{0}_{t} = \left((\widetilde{A}_0+\widetilde{B}_{0} \Xi) \,X^{0}_{t} + \mb{B}_0\, u^0_t +\widetilde{B}_{0}\varsigma(t)\right )dt + \Sigma_{0} dW^{0}_{t}.
 \end{equation}
 They still have the same form but matrices are reduced to the following.
\begin{gather}
 \widetilde{A}_0 = \left[ \begin{array}{cc}
1 & 1 \\
1 &  2
\end{array} \right]\!,
\quad \mb{B}_0=\left[ \begin{array}{c} 1 \\ 0  \end{array}\right]
\!,
\quad
 \widetilde{B}_0 = \left[ \begin{array}{c} 0 \\ 1  \end{array}\right]\!, \Sigma_0= \left[ \begin{array}{cc}
\sigma_0 & 0 \\
0 &  0
\end{array} \right],
W^{0}_{t} = \left[ \begin{array}{c}
w^{0}_{t}\\
0
\end{array} \right]\!.
\end{gather}

 \begin{equation}
\widetilde{A}_0+\widetilde{B}_{0} \Xi= \left[ \begin{array}{cc}
1 & 1 \\
1+\Xi_1(t) &  2+\Xi_2(t)
\end{array} \right]
 \end{equation}
 The limiting cost functional (see \eqref{costinmaj}) of major agent is changed to 
 \begin{align} 
J_{0}^{\infty}(u^{0})&=\mathbb{E}\Bigg[ \mathrm{exp}(\frac{1}{2}\int_{0}^{T}\left<\mathbb{Q} _{0}X_{s}^{0},X_{s}^{0}\right>+(u^{0}_s)^2dt) \Bigg] 
,\quad  \mathbb{Q}_0 = \left[ \begin{array}{cc}
1 & -1 \\
-1&  1
\end{array} \right]
\end{align} 
By following the same method we have 
 \begin{equation}\label{optconmajorsijsimp}
 u^{0,*}_{t} = -\mathbb{B}_0^\intercal \big(\Pi_0(t) X^{0}_{t} + s_0(t) \big)=-(\Pi_{0,11}x^{0}_{t}+\Pi_{0,12}\bar{x}_{t}+s_0(t))
 \end{equation}
 and
  \begin{multline} \label{eqmajorsimp}
 \Big(\dot{\Pi}_{0}(t)+\mathbb{Q}_{0}+(\widetilde{A}_0+\widetilde{B}_{0}\Xi(t))^{\intercal}\Pi_{0} (t)+\Pi_{0} (t)(\widetilde{A}_0+\widetilde{B}_{0}\Xi(t))-\Pi_{0}(t)\mb{B}_0\mb{B}_0^{\intercal}\Pi_{0}(t)+\delta_{0}\Pi_{0}(t) \Sigma_{0}\Sigma_{0} ^{\intercal}\Pi_{0}(t)\Big)X^{0}_{t} \\+\Big(\dot{s}_{0}(t)+( (\widetilde{A}_0+\widetilde{B}_{0}\Xi(t))^{\intercal}-\Pi_{0}(t)\mb{B}_0\mb{B}_0^{\intercal}+\delta_{0}\Pi_{0}(t) \Sigma_{0}\Sigma_{0} ^{\intercal}  )s_{0}(t)+\Pi_{0} (t)\widetilde{B}_{0}\varsigma(t)\Big)=0.    
\end{multline}
In this case, $\Pi_0(t)$ is just a $2 \times 2$ matrix, and $s_0(t)$ is a $\mbR^2$ vector. Now we move to minor agent's problem, see  \eqref{dyminet}-\eqref{eqminor}.

\begin{equation}
dX^{i}_t=( \grave{A}_{k}(t)X^{i}_t+\mathbb{B}_{k}u^{i}_{t}+\grave{M}_{k}(t))dt+\Sigma _{k}dW^{i}_{t}  \end{equation}
where
\begin{gather}
 \grave{A}_k(t) = \left[ \begin{array}{cc}1 ,1,1\\ 0 &\widetilde{A}_0+\widetilde{B}_{0} \Xi+\mb{B}_0\mb{B}_0^{\intercal}\Pi_{0}(t)) \end{array} \right], \,\,\,
  \grave{M}_k(t) = \left[ \begin{array}{c} 0 \\ -\mb{B}_0\mb{B}_0^{\intercal}s_{0}(t)+\widetilde{B}_{0}\varsigma(t) \end{array}\right].\label{sysMatMinorsimp}
\end{gather}

\begin{equation}
\widetilde{A}_0+\widetilde{B}_{0} \Xi+\mb{B}_0\mb{B}_0^{\intercal}\Pi_{0}(t))= \left[ \begin{array}{cc}1+\Pi_{0,11} ,1+\Pi_{0,12}\\ 1+\Xi_1(t) ,2+\Xi_2(t) \end{array} \right], \,\,\,
  -\mb{B}_0\mb{B}_0^{\intercal}s_{0}(t)+\widetilde{B}_{0}\varsigma(t) = \left[ \begin{array}{c} s_{0} \\ \varsigma(t)\end{array}\right]
\end{equation}
The limiting dynamic is given by 
\begin{equation}
J_{i}^{\infty}(u^{i})=\mathbb{E}\Bigg[ \mathrm{exp}(\frac{1}{2}\int_{0}^{T}\left<\mathbb{Q} X_{s}^{i},X_{s}^{i}\right>+(u^{i}_s)^2dt) \Bigg] 
,\quad  \mathbb{Q} =\begin{bmatrix}
1 &-1  &-1 \\ 
 -1&1  & 1\\ 
 -1& 1 & 1
\end{bmatrix}
\end{equation}
Similarly, we have
\begin{equation} \label{optconminorsimp}
 u^{i,*}_t = - \mathbb{B}^\intercal \left(\Pi(t) \,X^{i,*}_t + s(t) \right)=-(\Pi_{11}x^{i}_{t}+\Pi_{12}x^{0}_{t}+\Pi_{13}\bar{x}_{t}+s(t)) ,
\end{equation}
and
\begin{multline} \label{eqminorsimp}
\Big(\dot{\Pi}(t)+\mathbb{Q}+\grave{A}(t)^{\intercal}\Pi (t)+\Pi (t)\grave{A}(t)-\Pi(t)\mb{B} \mb{B}^{\intercal}\Pi(t)+\delta_{k}\Pi_{k}(t) \Sigma_{k}\Sigma_{k} ^{\intercal}\Pi_{k}(t)\Big)X^{i}_{t} \\+\Big(\dot{s}(t)+( \grave{A}(t)^{\intercal}-\Pi(t)\mb{B}\mb{B}^{\intercal}+\delta\Pi(t) \Sigma\Sigma ^{\intercal}  )s(t)+\Pi (t)\grave{M}(t)\Big)=0.   
\end{multline}
$\Pi(t)$ is a $3 \times 3$ matrix, and $s(t)$ is a $\mbR^3$ vector. From \eqref{optconminorsimp}, we can derive
\begin{equation}
\bar{u}_{t}=-((\Pi_{11}+\Pi_{13})\bar{x}_{t}+\Pi_{12}x^{0}_{t}+s_1(t))
\end{equation}
Thus
\begin{equation}
\Xi(t)=\begin{bmatrix}
-\Pi_{12}(t) & -\Pi_{11}(t)-\Pi_{13}(t),
\end{bmatrix},    
\varsigma(t)=s_1(t)
\end{equation}
Then we can recover the mean filed system in Theorem \eqref{Nash Equilibrium}. 
In particular, the mean field equation \eqref{MF_uperturbed} is reduced to 
\begin{equation}
d\bar{x}_{t} = \left(\ (2-\Pi_{11}-\Pi_{13})\bar{x}_{t} + (1-\Pi_{12})x^{0}_{t}  \right)dt,    
\end{equation}

\section{Conclusions}\label{sec_conclusion}
In this paper, we began by developing a variational analysis to solve LQG risk-sensitive optimal control problems, incorporating risk sensitivity through exponential cost functionals. We then extended our investigation to LQG risk-sensitive MFGs that involve a major agent and a large number of minor agents. The results obtained contribute to a deeper understanding of the implications of risk sensitivity in both LQG single-agent systems and MFGs. 

\bibliographystyle{elsarticle-num-names} 

\bibliography{sample}

\end{document}